\documentclass[11pt,oneside,reqno]{amsart}
\usepackage{amsmath}
\usepackage{amssymb}
\usepackage{graphicx}
\usepackage{tikz-cd}
\usepackage[T1]{fontenc}
\usepackage[utf8]{inputenc}

\usepackage[colorlinks=true,
            citecolor={red!60!black},
            linkcolor=blue,
            urlcolor=blue]{hyperref}

\usepackage[all]{xy}
\usepackage{xypic}
\usepackage{mathtools}

\usepackage{xspace}
\usepackage{bm}
\usepackage{amsmath}
\usepackage{amstext}
\usepackage{amsfonts}
\usepackage[mathscr]{euscript}
\usepackage{amscd}
\usepackage{latexsym}
\usepackage{amssymb}
\usepackage{enumerate}
\usepackage{xcolor}

\usepackage{anysize}
\marginsize{2.5cm}{2.5cm}{2.5cm}{2.5cm}

\theoremstyle{plain}
    \newtheorem{theorem}{Theorem}[section]
    \newtheorem{proposition}[theorem]{Proposition}
    \newtheorem{lemma}[theorem]{Lemma}

    \newtheorem{corollary}[theorem]{Corollary}
    
    \newtheorem{subsec}[theorem]{}
    
    \newtheorem*{thma}{Theorem A}
    \newtheorem*{thmb}{Theorem B}
    \newtheorem*{thmc}{Theorem C}
    \newtheorem*{thmd}{Theorem D}
    \newtheorem*{thme}{Theorem E}

\theoremstyle{definition}
    \newtheorem{construction}[theorem]{Construction}
    \newtheorem{definition}[theorem]{Definition}
    \newtheorem{example}[theorem]{Example}

\theoremstyle{remark}
        \newtheorem{remark}[theorem]{Remark}

\renewcommand{\thefigure}{\arabic{section}.\arabic{figure}}

\newenvironment{myeq}[1][]
{\stepcounter{figure}\begin{equation}\tag{\thefigure}{#1}}
{\end{equation}}

\newcommand{\hdim}{{\sf {hdim}}}

\newcommand{\secat}{{\sf {secat}}}

\newcommand{\TC}{{\sf {TC}}}

\newcommand{\cat}{{\sf {cat}}}

\newcommand{\RR}{{\mathbb {R}}}

\newcommand{\zcl}{{\sf {zcl}}}
\newcommand{\conn}{{\sf {conn}}}

\newcommand{\eat}[1]{}
\newcommand{\cp}{{\sf {cup}}}

\allowdisplaybreaks

\title[Generalized Milnor Manifold and Their TC]{On Generalized Milnor Manifolds\\
and Their Topological Complexity}

\author{Sarjick Bakshi}	
\address{Department of Mathematics and Statistics, IIT Kanpur, Kanpur 208016, India}
\email{sarjick91@gmail.com}

\author{Manas Mandal}
\address{Department of Mathematics and Statistics, IIT Kanpur, Kanpur 208016, India}
\email{manas.imsc@gmail.com}

\author{Subhankar Sau}
\address{Department of Mathematics and Statistics, IIT Kanpur, Kanpur 208016, India}
\email{subhankarsau18@gmail.com}
	
\thanks{ }

\begin{document}

\begin{abstract}
    We introduce generalized Milnor manifolds (GMM), extending the classical Milnor manifolds over $\mathbb{R}$, $\mathbb{C}$, and $\mathbb{H}$.
    We compute their integral cohomology algebras in the complex and quaternionic cases and their mod-$2$ cohomology algebras in the real case. We compare GMM with partial flag manifolds, investigate when they are homotopy equivalent, and obtain a necessary and sufficient condition in the complex and quaternionic cases. 
    We further prove that complex GMM are K\"ahler manifolds. As an application, we determine their higher topological complexities, obtaining exact values in the complex and quaternionic cases and bounds in the real case. Along the way, for a fibre bundle satisfying the Leray--Hirsch hypothesis, we establish a lower bound for the zero-divisor cup-length of the total space in terms of base and fibre.
\end{abstract}

\maketitle


\section{Introduction}

The classical \emph{Milnor manifolds} were introduced by Milnor \cite{Mil65} in 1965 in his study of the unoriented cobordism ring $\mathfrak{N}_*$.  
These manifolds arise as smooth hypersurfaces of codimension one in products of projective spaces, defined by
    \[
        \mathbb{C}H(m,n):= \big\{ (\mathbf{x},\mathbf{y}) \in \mathbb{C}P^m \times \mathbb{C}P^n \mid \textstyle\sum_{i=0}^{\min\{m,n\}} x_i y_i = 0 \big\}.
    \]
Since their introduction, Milnor manifolds have been studied extensively from various topological and geometric perspectives, see \cite{Mor99, Bro88, LJH22}.

In this paper, we introduce and study a natural generalization of these manifolds, including their real and quaternionic analogues. We call them \textit{generalized Milnor manifolds} (GMM).

Let $\mathbf{n} = (n_1, n_2, \ldots, n_\ell)$ be a weakly increasing sequence of natural numbers and $\mathbb F\in \{\mathbb R, \mathbb C, \mathbb H\}$. The generalized Milnor manifold associated to $\mathbf{n}$, denoted by $\mathbb{F}H(\mathbf{n})$, is defined as:
    \begin{equation*}
    \mathbb{F}H(\mathbf{n}) := 
    \Bigl\{ (\mathbf{x}^1, \mathbf{x}^2, \dots, \mathbf{x}^\ell) \in \prod_{t=1}^\ell \mathbb{F}P^{n_t} \;\Big|\; 
    \sum_{s=0}^{n_i} x^i_s x^j_s = 0 \;\text{for all } 1 \le i < j \le \ell \Bigr\},
    \end{equation*}
where $\mathbb{F}P^{n_t}$ is a $\mathbb F$-projective space of dimension $n_t$, and $\mathbf{x}^t=[x^t_0 : x^t_1 : \dots : x^t_{n_t}]\in\mathbb{F}P^{n_t}$.

Similar to the classical case, we establish the equivalence of two natural constructions of generalized Milnor manifolds: one via defining equations and another via iterated projectivizations (see Proposition~\ref{equivalence of definitions}).  
Using the latter description, we compute the cohomology algebra $H^*(\mathbb{F}H(\mathbf{n}); R)$ with coefficient ring $R=\mathbb{Z}_2$ when $\mathbb{F}=\mathbb{R}$ and $R=\mathbb{Z}$ when $\mathbb{F}=\mathbb{C}$ or $\mathbb{H}$.

\begin{thma}[Proposition~\ref{cohom of GMM}]\label{intro cohom of GMM}
    There is an isomorphism of graded algebras
        \[
        H^*(\mathbb{F}H(\mathbf{n});R)
        \cong 
        R[x_1,\dots,x_\ell]\big/\langle h_{n_1+1}(x_1),\, h_{n_2}(x_1,x_2),\, \dots,\, h_{n_\ell-\ell+2}(x_1,\dots,x_\ell)\rangle,
        \]
    where $\deg(x_i)=\mathrm{dim}_{\mathbb R}(\mathbb F)$ and the isomorphism is given by $c_1(\xi_i) \longmapsto -x_i.$
    Here $h_k(x_1,\dots,x_r)$ denotes the complete homogeneous symmetric polynomial of degree $k$ in the variables $x_1,\dots,x_r$.
\end{thma}

Our subsequent studies revolve around the following central observation (Proposition \ref{Prop_Kahler}):
    $$\text{\textit{The complex GMM are K\"ahler manifolds.}}$$
Using this observation, first we compare the GMM with partial flag manifolds, and completely characterize when a generalized Milnor manifold is homotopy equivalent to a partial flag manifold in complex and quaternionic cases. We also deal with the real case in Proposition \ref{prop_GMM=flag}.

In our notation, a partial flag manifold $\mathbb FG(\boldsymbol{\mu})$ for a sequence of natural numbers $\boldsymbol{\mu}=(m_1,\ldots,m_\wp)$ is the space consisting of all orthogonal decompositions $\mathbb F^m=V_1\oplus\cdots\oplus V_\wp$, where $\dim_{\mathbb F}(V_i)=m_i$ for $i=1,\ldots,\wp$. Now we present the result for complex and quaternionic cases.
\begin{thmb}[Proposition \ref{GMM neq flag}]
    Let $\mathbf{n} = (n_1, n_2, \ldots, n_\ell)$ and $\boldsymbol{\mu} = (m_1, m_2, \ldots, m_\wp)$ be two weakly increasing sequences of positive numbers and $\mathbb{F} \in \{\mathbb{C}, \mathbb{H}\}$. Then the generalized Milnor manifold $\mathbb{F}H(\mathbf{n})$ is homotopy equivalent to the partial flag manifold $\mathbb{F}G(\boldsymbol{\mu})$ if and only if:
    $$
        \ell = \wp - 1, \quad n_1 = n_2 = \cdots = n_\ell, \quad m_1 = m_2 = \cdots = m_{\wp-1} = 1, \quad \text{and} \quad m_\wp = n_\ell - \ell + 1.
    $$
\end{thmb}

On a different note, we construct a family of generalized Milnor manifolds interpolating between two partial flag manifolds of type $\mathbb{F}G(1,\ldots,1,m)$. Then we address the question of maximum possible dimension of a partial flag manifold that admits a certain ``nice'' embedding into a given generalized Milnor manifold $\mathbb FH(\mathbf n)$ (Proposition \ref{Prop_interpolation}).

This K\"ahlerness of complex GMM also leads us to the study of their \textit{topological complexity} ($\TC$). 
In recent years, $\TC$ and related numerical homotopy invariants have attracted considerable attention following the pioneering work of Farber \cite{Far03}. 
Topological complexity is sectional category of certain path space fibrations. The higher topological complexity $\TC_r(X)$ is a higher analogue of Farber's topological complexity introduced by Rudyak \cite{Rud10} that measures the complexity of motion planning between $r$ prescribed states.
Since determining the exact value of $\TC_r$ is generally difficult, one often relies on  bounds using zero-divisor-cup-length in cohomology algebra and dimensional arguments.

In our article, using the K\"ahler structure of GMM and \cite[Corollary 3.15]{BGRT14},
we obtain:

\begin{thmc}[Theorem~\ref{Prop_higher TC of complex GMM}]
    Let $\mathbb{F}H(\mathbf{n})$ be a generalized Milnor manifold of $\mathbb{F}$-dimension $D$ and $r \geq 2$. Then for $\mathbb{F}=\mathbb{C}$ and $\mathbb{H}$,
        $$\TC_r(\mathbb{F}H(\mathbf{n})) = rD.$$
\end{thmc}

This result recovers and extends earlier computations of topological complexity for classical complex Milnor manifolds (corresponding to the case $\mathbf{n}=(n_1,n_2)$), as obtained by Daundkar \cite{ND25} and Daundkar--Singh \cite{DS25}. The K\"ahlerness of complex GMM provides a natural and effective framework for studying topological complexity in this broader setting.

The study of topological complexity of generalized real Milnor manifolds presents an additional challenge, as K\"ahler structures are unavailable in this setting. To overcome this difficulty, we rely on the following general observation.

\begin{thmd}[Proposition \ref{Prop_zcl in Leray Hirsch setup}]
    Let $p:E\to B$ be a fibre bundle with fibre $F$ such that $H^*(F;\mathbb K)$ is finite-dimensional and the inclusion of each fibre into $E$ induces a surjection in cohomology with coefficients in a field $\mathbb K$. Then 
        $$\zcl_r(F) + \zcl_r(B) \leq \zcl_r(E).$$
\end{thmd}

This result provides a flexible tool for estimating the zero-divisor cup-length, and hence lower bounds for $\TC_r$, of total spaces in the Leray--Hirsch setting.  One can find the study of topological complexity of these kind of spaces in \cite{BNSS25}, \cite{BDS26} for generalized Bott manifolds and \cite{GGGL16} for partial flag manifolds. Here, Theorem D offers an alternative perspective on such studies, using appropriate fibre bundle structures.
In our article, it leads to the following result on topological complexity of real GMM.
\begin{thme}[Corollary \ref{Cor_tc of real gmm}]
    For a real generalized Milnor manifold $\RR H (\textbf{n})$ of dimension $D$ satisfying $n_i=2^{t_i}+i-1$ where $t_i \geq 0$ for $1 \leq i \leq \ell$, we have
    \begin{align*}
         2D- \ell \leq &\TC_r (\RR H(\textbf{n})) \leq 2D  \quad \text{for } r=2 \text{ and}\\
         &\TC_r(\RR H(\textbf{n}))=rD  \quad \text{for } r \geq3.
    \end{align*}
\end{thme}
 This improves the lower bounds for the topological complexity of certain classical real Milnor manifolds obtained in \cite[Theorem 2.9(2)]{DS25}.

\noindent \textbf{Organization of the paper.} 
In Section \ref{Sec_GMM}, we present two different constructions of generalized Milnor manifolds and establish their equivalence. We then obtain Borel-type presentations of their cohomology rings. Next, we prove that generalized complex Milnor manifolds are K\"ahler manifolds.
Section \ref{Sec_GMM and Flag mfld} is devoted to a comparison of the generalized Milnor manifold with partial flag manifolds. 
We construct a family of generalized Milnor manifolds interpolating between two partial flag manifolds in Section \ref{Sec_Interpolation}. 
Finally, in Section \ref{Sec_TC}, we focus on computing their higher topological complexities.


\section{Generalized Milnor Manifolds: Constructions and Basic Properties}\label{Sec_GMM}

In this section, we first recall classical Milnor manifolds from \cite{Mil65} and then introduce a natural generalization of them. We present two equivalent constructions of these manifolds and subsequently describe their cohomology algebras.

Let $\mathbb{F}$ denote the field $\mathbb{R}$ or $\mathbb{C}$ or the division algebra $\mathbb{H}$, and let $m,n$ be positive integers with $m\le n$. The classical \emph{Milnor manifold} is defined by    
   \begin{equation}\label{Eq_classical MM}
        \mathbb{F}H(m,n)
        \coloneqq
        \bigl\{ (\mathbf{x},\mathbf{y}) \in \mathbb{F}P^m \times \mathbb{F}P^n 
        \;\big|\; 
        \sum_{i=0}^{m} x_i\, y_i = 0
        \bigr\},
    \end{equation}
where $\mathbf{x}=[x_0:x_1:\cdots:x_m]\in\mathbb{F}P^m$ and 
$\mathbf{y}=[y_0:y_1:\cdots:y_n]\in\mathbb{F}P^n$.

In the case $\mathbb{F} = \mathbb{C}$, an alternative definition frequently appears in the literature, given by 
    \[
        \sum_{i=1}^{\min\{m,n\}} x_i \bar{y_i} = 0,
    \]
where $[x_0:x_1:\cdots:x_m]\xmapsto{\hspace{0.5em} \sigma \hspace{0.5em}}[\bar x_0:\bar x_1:\cdots:\bar x_m]$ is the standard conjugation on a complex projective space.
This defines a space homeomorphic to $\mathbb{C}H(m,n)$ via the map $\mathrm{id} \times \sigma : \mathbb{C}P^m \times \mathbb{C}P^n \to \mathbb{C}P^m \times \mathbb{C}P^n$. 

The classical Milnor manifold $\mathbb{F}H(m,n)$ can be described equivalently as the projectivization
    \[
        \mathbb{F}H(m,n) \cong \mathbb{P}\big(\xi^\perp \oplus \varepsilon_{\mathbb{F}}^{\,n-m}\big),
    \]
where $\xi$ denotes the canonical $\mathbb{F}$-line bundle over $\mathbb{F}P^m$ and $\xi^\perp$ is its orthogonal complement in the trivial bundle $\varepsilon_{\mathbb{F}}^{\,m+1}$. In particular, $\mathbb{F}H(m,n)$ is the total space of the fibre bundle
   \begin{equation}\label{Eq_classical MM fibre}
        \begin{tikzcd}
        	{\mathbb{F}P^{n-1}} & {\mathbb{F}H(m,n)} & {\mathbb{F}P^{m}}
        	\arrow["\iota", hook, from=1-1, to=1-2]
        	\arrow["p", two heads, from=1-2, to=1-3]
        \end{tikzcd}
    \end{equation}
where $p$ is the natural projection defined by $(\mathbf x, \mathbf y)\mapsto \mathbf x.$

We now introduce two natural generalizations of the definitions of Milnor manifold as in \eqref{Eq_classical MM} and \eqref{Eq_classical MM fibre}. Then we show that these two constructions are homeomorphic. 

Unless otherwise stated, by a \emph{sequence} $\mathbf{n}$ we mean a weakly increasing finite sequence of natural numbers
$\mathbf{n}=(n_1,n_2,\ldots,n_\ell)$.
Such sequences will be denoted by bold symbols, for example $\mathbf{n}$, $\mathbf{m}$, $\boldsymbol{\nu}$, and $\boldsymbol{\mu}$. Throughout this article, we generally use $\mathbf{m}$ and $\mathbf{n}$ in the context of generalized Milnor manifolds, and $\boldsymbol{\mu}$ and $\boldsymbol{\nu}$ in the context of partial flag manifolds. Whenever an entry $n_i$ appears consecutively $t$ times in a sequence, we write $n_i^t$ to denote this block of repeated entries.

\begin{definition}\label{defn of GMM equations}
    The \emph{generalized Milnor manifold} (GMM), denoted by $\mathbb{F}H(\mathbf{n})$ for a weakly increasing finite sequence $\mathbf{n}=(n_1,n_2,\ldots,n_\ell)$, is defined as
        \begin{equation}\label{GMM defn}
            \mathbb{F}H(\mathbf{n}) := 
            \Bigl\{ (\mathbf{x}^1, \mathbf{x}^2, \dots, \mathbf{x}^\ell) \in \prod_{t=1}^\ell \mathbb{F}P^{n_t} \;\Big|\; 
            \sum_{s=0}^{n_i} x^i_s x^j_s = 0 \;\text{for all } 1 \le i < j \le \ell \Bigr\},
        \end{equation}
    where $\prod_{t=1}^\ell \mathbb{F}P^{n_t}$ denotes the product of the projective spaces and $\mathbf{x}^t=[x^t_0 : x^t_1 : \dots : x^t_{n_t}]$ represents a point in $\mathbb{F}P^{n_t}$.
\end{definition}

\begin{remark}
    If $\ell>n_{\ell}+1$, then $\mathbb{F}H(\mathbf n)$ is empty, since it is impossible to choose pairwise nonzero orthogonal vectors $\mathbf x^{1},\dots,\mathbf x^{\ell+1}$ in $\mathbb F^{n_{\ell}+1}$. \qed
\end{remark}

We now describe an alternative construction of the generalized Milnor manifold $\mathbb{F}H(\mathbf{n})$ as an iterated projectivization of vector bundles.

\begin{construction}\label{construction of GMM fibre bundle}
    Consider the following diagram of iterated projectivizations:
    \[\begin{tikzcd}[row sep=1ex]
    	& {E_{\ell-1}} & \cdots & {E_2} & {E_1} & {E_0} \\
    	{B_{\ell}} & {B_{\ell-1}} & \cdots & {B_2} & {B_1} & {B_0=\{*\}}
    	\arrow[from=1-2, to=2-2]
    	\arrow[from=1-4, to=2-4]
    	\arrow[from=1-5, to=2-5]
    	\arrow[from=1-6, to=2-6]
    	\arrow["{p_{\ell}}", two heads, from=2-1, to=2-2]
    	\arrow["{p_{\ell-1}}", two heads, from=2-2, to=2-3]
    	\arrow["{p_3}", two heads, from=2-3, to=2-4]
    	\arrow["{p_2}", two heads, from=2-4, to=2-5]
    	\arrow["{p_1}", two heads, from=2-5, to=2-6]
    \end{tikzcd}\]
    where we set \begin{enumerate}[(i)]
        \item  $B_0=\{*\}$, a one-point space, and $E_0=\mathbb{F}^{n_1+1}$;
    
        \item $E_{i} \coloneqq \xi_{i}^{\perp} \oplus \varepsilon_{\mathbb{F}}^{n_{i+1}-n_{i}}$ for $i=1,2,\ldots,\ell-1$;
    
        \item $B_{i}:=\mathbb FP(E_{i-1})$, projectivization of the bundle $E_{i-1}$, for $i=1,2,\ldots,\ell$.
    \end{enumerate}
\end{construction}

For $\ell=2$, the space $B_2 \cong \mathbb{F}H(n_1,n_2)$ recovers the classical Milnor manifold. More generally, the following proposition shows that the above construction yields the generalized Milnor manifold defined in Definition~\ref{defn of GMM equations}.

\begin{proposition}\label{equivalence of definitions}
    The space $B_\ell$ constructed in Construction~\ref{construction of GMM fibre bundle} is homeomorphic to the generalized Milnor manifold $\mathbb{F}H(\mathbf{n})$ defined in Definition~\ref{defn of GMM equations}.
\end{proposition}

\begin{proof}
First, we observe that any point $\mathbf{x}_\ell\in B_\ell$ uniquely determines points 
$\mathbf{x}_i\in B_i$, $i=1,2,\ldots,\ell$, such that 
$\mathbf{x}_i\in F_{\mathbf{x}_{i-1}}$, where $F_b$ denotes the fibre over the point
$b$ of the appropriate projective bundle in the definition of $B_\ell$.
Let $\{e_i\}_{i=1}^\ell$ be the standard basis of $\mathbb{F}^{n_\ell+1}$, and
identify $\mathbb{F}^{n_i+1}$ with the subspace of $\mathbb{F}^{n_\ell+1}$
generated by $\{e_j\}_{j=1}^{i+1}$.

Note that $\mathbf{x}_1\in B_1=\mathbb{F}P^{n_1}$ is an $\mathbb{F}$-line in
$\mathbb{F}^{n_1+1}$, $\mathbf{x}_2$ is an $\mathbb{F}$-line in
$\mathbb{F}^{n_2+1}$ which is orthogonal to the $\mathbb{F}$-line
$\mathbf{x}_1\subseteq \mathbb{F}^{n_1+1}\subseteq \mathbb{F}^{n_2+1}$, and
$\mathbf{x}_3$ is an $\mathbb{F}$-line in $\mathbb{F}^{n_3+1}$ orthogonal to both
$\mathbf{x}_1$ and $\mathbf{x}_2$.
Proceeding similarly, $\mathbf{x}_\ell$ determines a unique sequence of
$\mathbb{F}$-lines $\mathbf{x}_i\subset \mathbb{F}^{n_i+1}$ for
$i=1,2,\ldots,\ell$, such that $\mathbf{x}_i\perp \mathbf{x}_j$ for $i\neq j$.
These orthogonality conditions can alternatively be expressed by 
    \[
    \sum_{s=0}^{\min\{n_i,n_j\}} x_s^i x_s^j = 0
    \quad \text{for } i\neq j .
    \]
This gives a continuous bijective correspondence between $B_\ell$ and
$\mathbb{F}H(\mathbf n)$. Since $B_\ell$ is compact and
$\mathbb{F}H(\mathbf n)$ is Hausdorff, the correspondence is a homeomorphism.
\end{proof}

It follows from the iterated projective bundle construction that the generalized Milnor manifold $\mathbb{F}H(\mathbf{n})$ has $\mathbb{F}$-dimension
   \begin{equation}\label{Eq_dimension of GMM}
        D \coloneqq \sum_{j=1}^{\ell} n_j - \frac{\ell(\ell-1)}{2}.
    \end{equation}
We will use the notation $D$ for this quantity throughout the article. In particular, the real dimension of $\mathbb{F}H(\mathbf{n})$ is $d_{\mathbb{F}} \cdot D$, where $d_{\mathbb{F}} = \dim_{\mathbb{R}} \mathbb{F}$.

Moreover, for each $i=1,\dots,\ell$, the space $B_i$ appearing in Construction~\ref{construction of GMM fibre bundle} is itself a generalized Milnor manifold corresponding to the finite sequence $(n_1,\dots,n_i)$. 

\begin{remark}\label{remarks}
\begin{enumerate}[(i)]

\item If $n=n_1=\cdots=n_\ell$ and $\ell \le n$, then the generalized Milnor manifold $\mathbb{F}H(\mathbf{n})$ is homeomorphic to the almost complete flag manifold $\mathbb{F}G(\boldsymbol{\nu})$ of type $\boldsymbol{\nu} = (1^\ell,\, n-\ell+1),$ which consists of all orthogonal decompositions of $\mathbb F^{n+1}$ into $\ell $ many $\mathbb F$-lines and a subspace of dimension $n-\ell+1.$

\item If $n_k < n_{k+1}$, then the $\mathbb{F}P^{\,n_{k+1}-1}$-bundle 
$p_k : B_{k+1} \to B_k$ admits a global section. Indeed, since the trivial bundle 
$\varepsilon_{\mathbb{F}}^{\,n_{k+1}-n_k}$ has positive rank, any nonzero vector 
$v \in \mathbb{F}^{\,n_{k+1}-n_k}$ defines a section $s_{[v]} : B_k \to B_{k+1}$ given by
    \[
    s_{[v]}(b) = [v] \in 
    \mathbb{F}P(F_b)
    = \mathbb{F}P\big((\xi_k^\perp)_b \oplus \mathbb{F}^{\,n_{k+1}-n_k}\big)
    \cong \mathbb{F}P^{\,n_{k+1}-1},
    \quad b \in B_k,
    \]
where $F_b$ denotes the fibre over $b$.

\item Iterating the homotopy long exact sequence for the fiber bundle 
$p_k : B_{k+1} \to B_k$ (see Construction~\ref{construction of GMM fibre bundle}), we obtain
    \[
    \pi_1\big(\mathbb{F}H(\mathbf{n})\big) =
    \begin{cases}
    \mathbb{Z}_2^{\,\ell}, & \text{if } \mathbb{F} = \mathbb{R},\\
    0, & \text{if } \mathbb{F} = \mathbb{C}, \mathbb{H}.
    \end{cases}
    \]\qed
\end{enumerate}
\end{remark}

A generalized Milnor manifold $\mathbb{F}H(\mathbf n)$ admits several natural fibre bundle structures in which both the base and the fibre are themselves generalized Milnor manifolds. To describe such fibre bundle structures, let $\mathbf n_b=(n_1,n_2,\ldots,n_b)$ be an initial subsequence of the weakly increasing sequence $\mathbf n=(n_1,n_2,\ldots,n_\ell)$. Associated to $\mathbf n_b$, we have the generalized Milnor manifold $\mathbb{F}H(\mathbf n_b)$.

We now construct a fibre bundle structure on $\mathbb{F}H(\mathbf n)$ with base $\mathbb{F}H(\mathbf n_b)$. Using Definition~\ref{defn of GMM equations}, for any $\mathbf x=(\mathbf x^1,\mathbf x^2,\ldots,\mathbf x^b)\in \mathbb{F}H(\mathbf n_b)$, define a subspace $F_{\mathbf x}$ of $\mathbb FH(\mathbf n)$ as
\begin{align*}
    F_{\mathbf x}:=&\Big\{(\mathbf{y}^1, \mathbf{y}^2, \dots, \mathbf{y}^\ell)\in \mathbb FH(\mathbf n)\;\mid \; \mathbf x^i=\mathbf y^i\in \mathbb FP^{n_i}\text{ for }i=1,2,\ldots, b \Big\}\\
    =&\Big\{(\mathbf{y}^1, \mathbf{y}^2, \dots, \mathbf{y}^\ell)\in\prod_{t=1}^\ell \mathbb FP^{n_t}\;\mid \; \mathbf y^i\perp\mathbf y^j\text{ for }1\le i<j\le \ell \text{ and } \mathbf x^k=\mathbf y^k\text{ for }1\le k\le b \Big\}\\
    =&\Big\{(\mathbf y^{b+1},\mathbf y^{b+2},\ldots,\mathbf y^\ell)\in \prod_{t=b+1}^\ell \mathbb FP^{n_t}\;\mid \; \mathbf y^i\perp\mathbf y^j ,\; \mathbf y^i\perp\big(\bigoplus_{k=1}^{b}\mathbf x^k\big)\text{ for }b< i<j\le \ell\Big\},
\end{align*} which is homeomorphic to
\begin{align*}
    \mathbb FH(\mathbf n_f)=\Big\{(\mathbf y^{b+1},\mathbf y^{b+2},\ldots,\mathbf y^\ell)\in \prod_{t=b+1}^\ell \mathbb FP^{n_t-b}\;\mid \; \mathbf y^i\perp\mathbf y^j \text{ for }b< i<j\le \ell\Big\},
\end{align*} where the sequence $\mathbf n_f=(n_{b+1}-b, n_{b+2}-b,\ldots, n_\ell-b).$
This gives the fibre bundle structure
    \begin{equation}\label{GMM bundle over GMM}
        \mathbb FH(\mathbf n_f)\hookrightarrow \mathbb FH(\mathbf n)\twoheadrightarrow \mathbb FH(\mathbf n_b).
    \end{equation}

Replacing $\mathbf n$ by $\mathbf n_f$ or $\mathbf{n}_b$, we can obtain  fibre bundle structures on $\mathbb{F}H(\mathbf n_f)$ and $\mathbb{F}H(\mathbf n_b)$ analogous to \eqref{GMM bundle over GMM}, with both the fibre and the base being generalized Milnor manifolds. In this way, iterating this construction yields the following proposition.

\begin{proposition}\label{Prop_GMM as iterated fibre bundle}
    Let $\mathbb{F}H(\mathbf n)$ be the generalized Milnor manifold associated to a weakly increasing sequence
    $\mathbf n=(n_1,n_2,\ldots,n_\ell)$, and let
    $\lambda=(\lambda_1,\lambda_2,\ldots,\lambda_l)$ be a partition of $\ell$.
    Then $\mathbb{F}H(\mathbf n)$ can be expressed as an iterated fibre bundle, with each fibre a generalized Milnor manifold, as follows
\[\begin{tikzcd}
	{\mathbb FH(\mathbf n)=\mathbb FH(\mathbf n_{\lambda_l})} & {\mathbb FH(\mathbf n_{\lambda_{l-1}})} & \cdots & {\mathbb FH(\mathbf n_{\lambda_2})} & {\mathbb FH(\mathbf n_{\lambda_1})} & \{*\}
	\arrow[two heads, from=1-1, to=1-2]
	\arrow[two heads, from=1-2, to=1-3]
	\arrow[two heads, from=1-3, to=1-4]
	\arrow[two heads, from=1-4, to=1-5]
    \arrow[two heads, from=1-5, to=1-6],
\end{tikzcd}\]
    where $\mathbb{F}H(\mathbf n_{\lambda_i})$ denotes the total space of a fibre bundle over $\mathbb{F}H(\mathbf n_{\lambda_{i-1}})$, whose fibre is given by
\[
\mathbb FH\big(n_{\lambda_{i-1}}-\theta_i,\; n_{\lambda_{i-1}+1}-\theta_i,\; \ldots\;,\; n_{\lambda_i}-\theta_i\big),\quad \theta_i:= \sum_{t=1}^i\lambda_i.
\]
\end{proposition}


\subsection{Cohomology of generalized Milnor manifolds}

In this subsection, we give a Borel-type presentation of the cohomology ring $H^*(\mathbb{F}H(\mathbf{n});R)$ of generalized Milnor manifolds. Here $R=\mathbb{Z}_2$ if $\mathbb{F}=\mathbb{R}$, and $R=\mathbb{Z}$ otherwise.

Let $p:\mathbb{F}P(E)\to B$ denote the projectivization of an $n$-rank $\mathbb{F}$-vector bundle $E\to B$, where $\mathbb{F}\in \{\mathbb{R}, \mathbb{C}, \mathbb{H}\}$. We write $c_i(E)$ for the $i$-th characteristic class of $E$, namely the $i$-th Stiefel--Whitney class when $\mathbb{F}=\mathbb{R}$, the $i$-th Chern class when $\mathbb{F}=\mathbb{C}$, and the $i$-th Pontryagin class when $\mathbb{F}=\mathbb{H}$. Here, for a quaternionic vector bundle $E$ of rank $n$, by Pontryagin class of $E$, we mean the Pontryagin class of the underlying real vector bundle $E_{\mathbb R}$ of rank $4n.$

We recall the construction of generalized Milnor manifolds from Construction~\ref{construction of GMM fibre bundle}. Following the notation introduced there, we obtain the following theorem.

\begin{theorem}\label{cohom of GMM}
There is an isomorphism of graded algebras
    \[
    H^*(\mathbb{F}H(\mathbf{n});R)
    \cong 
    R[x_1,\dots,x_\ell]\big/\langle h_{n_1+1}(x_1),\, h_{n_2}(x_1,x_2),\, \dots,\, h_{n_\ell-\ell+2}(x_1,\dots,x_\ell)\rangle,
    \]
where $\deg(x_i)=\mathrm{dim}_{\mathbb R}(\mathbb F)$ and the isomorphism is given by $c_1(\xi_i) \longmapsto -x_i.$
Here $h_k(x_1,\dots,x_r)$ denotes the complete homogeneous symmetric polynomial of degree $k$ in the variables $x_1,\dots,x_r$.
\end{theorem}

\begin{proof}
We proceed by induction on $\ell$, the number of projectivizations. The base case $\ell=1$ is trivial: $\mathbb{F}H(n_1) = \mathbb{F}P^{n_1}$, and the relation is $x_1^{n_1+1} = 0$, which is exactly $h_{n_1+1}(x_1) = 0$.

Assume the result holds for $B_{\ell-1}$. By construction, $B_\ell = \mathbb{F}P(E_{\ell-1})$ where $E_{\ell-1} = \xi_{\ell-1}^\perp \oplus \varepsilon^{n_\ell - n_{\ell-1}}$. By the Leray--Hirsch theorem \cite[Theorem 4D.1]{Hat}, $H^*(B_\ell; R)$ is a free $H^*(B_{\ell-1}; R)$-module generated by $\{1, x_{\ell}, \ldots, x_{\ell}^k\}$ where $k = n_{\ell} - \ell +1$. The generator $x_\ell = -c_1(\xi_\ell)$ satisfies the relation
    \begin{equation}\label{eq_leray_relation}
        \sum_{j=0}^{n_\ell - \ell +1} c_j(E_{\ell-1}) x_\ell^{n_\ell-\ell+1-j} = 0.
    \end{equation}
To compute the characteristic classes $c_j(E_{\ell-1})$, we use the splitting principle. By construction, the bundle $\varepsilon_{\mathbb{F}}^{n_\ell+1}$ admits a decomposition $\varepsilon_{\mathbb{F}}^{n_{\ell}+1} \cong \xi_1 \oplus \xi_2 \oplus \cdots \oplus \xi_{\ell} \oplus \xi_{\ell}^{\perp},$ which follows inductively from the definition of the bundles $E_i$.
Setting $x_i = -c_1(\xi_i)$ for all $i$ and using the Whitney sum formula, the total characteristic class of the complement bundle is given by
    \begin{equation*}
    c(\xi_\ell^\perp) = \prod_{i=1}^\ell (1 - x_i)^{-1}.
    \end{equation*}
Recall that the generating function for the complete homogeneous symmetric polynomials $h_k$ in $r$ variables is defined by $\sum_{k=0}^\infty h_k(x_1, \dots, x_r) t^k = \prod_{i=1}^r (1 - x_i t)^{-1}$. Evaluating at $t=1$, we get
    \begin{equation}\label{Eq_ h_j}
    c_j(\xi_{l-1}^\perp) = h_j(x_1, \dots, x_{l-1}).
    \end{equation}
Now, consider the bundle $E_{\ell-1} = \xi_{\ell-1}^\perp \oplus \varepsilon_{\mathbb F}^{n_\ell - n_{\ell-1}}$. We have $c(E_{\ell-1}) = c(\xi_{\ell-1}^\perp)$.
Thus, $c_j(E_{\ell-1}) = h_j(x_1, \dots, x_{\ell-1})$, by \eqref{Eq_ h_j}. Substituting this into the relation \eqref{eq_leray_relation}, we obtain
    \begin{equation}
    \sum_{j=0}^{n_\ell - (\ell-1)} h_j(x_1, \dots, x_{\ell-1}) x_\ell^{n_\ell - (\ell-1) - j} = 0.
    \end{equation}
By the standard recurrence relation for complete homogeneous polynomials, $h_k(x_1, \dots, x_r) = \sum_{j=0}^k h_j(x_1, \dots, x_{r-1}) x_r^{k-j}$, the sum above is precisely $h_{n_\ell - \ell + 1}(x_1, \dots, x_\ell)$. This completes the inductive step and the proof.
\end{proof}

For a nonzero nilpotent element $x$ in a ring $R$, its \emph{height} is defined as the positive integer $n$ such that $x^n \neq 0$ and $x^{n+1} = 0$. We have the following remark about the heights of the generators $x_i$ in $H^*(\mathbb FH(\mathbf n);R).$

\begin{remark}\label{remark heights of generators}
    The first $i$ many vector bundle relations appearing in the construction of $\mathbb FH(\mathbf n)$, $\xi_i\oplus \xi_i^\perp\cong p_i^*(E_{i-1}),\; i=1,2,\ldots,i$, combine into $\xi_1\oplus \xi_2\oplus \cdots \oplus \xi_i\oplus\xi_i^\perp\cong \mathbb F^{n_i+1}.$ Taking total characteristic class both sides appropriately to $\mathbb F$, we get
    \begin{align}
        &\prod_{k=1}^i (1-x_k)\cdot \big(1+c_1(\xi_i^\perp)+\cdots +c_{n_i-i+1}(\xi_i^\perp)\big)=1, \nonumber
        \\
        \implies  &\prod_{k=1}^{i-1} (1-x_k)\cdot \big(1+c_1(\xi_i^\perp)+\cdots +c_{n_i-i+1}(\xi_i^\perp)\big)= \big (1+x_i+x_i^2+\cdots\big)\label{Eq xi^ni+1 zero}.
    \end{align} 
    Using the Leray--Hirsch theorem, we see that the left-hand side of \eqref{Eq xi^ni+1 zero} is nonzero in degree $n_i$ and vanishes in all higher degrees. Consequently, the right-hand side must satisfy
    $x_i^{n_i}\neq 0$ and $x_i^{n_i+1}=0$.
    Therefore, the height of $x_i$ is $n_i$ for each $i$. \qed
\end{remark}

\subsection{Generalized complex Milnor manifolds are K\"ahler}\label{Subsec_GMM is Kahler}
Our goal in this subsection is to realize the generalized complex Milnor manifold $\mathbb{C}H(\textbf{n})$ as a compact K\"ahler manifold for any weakly increasing finite sequence $\textbf{n}$ of natural numbers. As an application, this provides us with a concrete foundation to study topological complexity of a generalized Milnor manifold.

A symplectic manifold $(M,\omega)$ equipped with an integrable almost complex structure $J$ is called a \emph{K\"{a}hler manifold} if $J$ is compatible with the symplectic form $\omega$. 
Throughout this subsection, we set $\mathbb{F}=\mathbb{C}$ and we simplify the notations by $\mathbb{P}^d \coloneq \mathbb{C}P^d$ for a $d$-dimensional complex projective spaces and $\mathbb{P}(E) \coloneq \mathbb{C}P(E)$ for a complex projectivization of a complex vector bundle $E$. 

We begin with a general remark about complements of subbundles in a vector bundle. Let $E \to X$ be a vector bundle over a paracompact base space $X$ and $W \subset E$ a subbundle. The paracompactness of $X$ ensures a Riemannian metric on $E$, that is, a continuous family of inner products on the fibers of $E$. With respect to this metric, each fiber admits an orthogonal decomposition
    $$E_b = W_b \oplus W_b^{\perp}.$$
These orthogonal complements vary continuously with $b$ and therefore assemble to form a subbundle $W^{\perp} \subset E$. The orthogonal projection $ E \longrightarrow W^{\perp}$ is fiberwise linear and has kernel $W$. Consequently, it induces a bundle isomorphism
    $$E/W \cong W^{\perp}.$$
In particular, the total spaces are homeomorphic. Indeed, they are linearly homeomorphic along the fibers, while the base space remains the same. Note that this construction is not canonical as it depends on the choice of the Riemannian metric. 
Therefore, replacing quotient bundles by orthogonal complement bundles while constructing $B_\ell$ in the Construction~\ref{construction of GMM fibre bundle} leads to the same spaces up to homeomorphism. In what follows, we freely use both descriptions interchangeably.

For basic definitions in complex geometry, we refer to the book of Huybrechts \cite{Huy05} and Voisin \cite{Voisin2002}. 
The \emph{tautological line bundle} on $\mathbb{P}^r$ is defined by
    $$\mathcal{O}_{\mathbb{P}^r}(-1)=    \{(l,z) \in \mathbb{P}^r \times \mathbb{C}^{r+1} \mid z \in l \},$$
and we recall from \cite[Proposition 2.2.6]{Huy05} that $\mathcal{O}_{\mathbb{P}^r}(-1)$ is a holomorphic line bundle over $\mathbb{P}^r$.
For a complex manifold $X$ and a holomorphic vector bundle $E \to X$ of rank $r$, we denote the associated projective bundle by
    $$\pi : \mathbb{P}(E) \to X.$$
We recall from \cite[Example 2.4.9]{Huy05} and \cite[\S9.1]{EisHar16} that locally a point of $\mathbb{P}(E)$ correspond to pair $(x,l)$ with $x \in X$ and $l$ a one-dimensional subspace of the fiber $E_x$ of $E$. 
The canonical holomorphic line subbundle $\mathcal{O}_{\mathbb{P}(E)}(-1) \subset \pi^*E$ is called the \emph{relative tautological holomorphic line bundle}, whose fiber at a point $(x,l) \in \mathbb{P}(E)$ is the line $\mathcal{O}_{\mathbb{P}(E)}(-1)_{(x,l)} = l \subset E_x.$
The inclusion gives rise to the \emph{universal exact sequence}
    $$0 \longrightarrow \mathcal{O}_{\mathbb{P}(E)}(-1) \longrightarrow \pi^*E \longrightarrow Q \longrightarrow 0,$$
where $Q$ is the universal quotient bundle of rank $r-1$.

We now revisit Construction~\ref{construction of GMM fibre bundle} of the generalized Milnor manifold using iterative projectivization of holomorphic vector bundles. Let $\mathbf{n}=(n_1,n_2,\ldots,n_\ell)$ be a weakly increasing sequence of positive integers. Set $B_1 := \mathbb{P}^{n_1}$. Let
    $$0 \longrightarrow \mathcal{O}_{B_1}(-1)\longrightarrow \mathcal{O}_{B_1}^{\oplus (n_1+1)} \longrightarrow Q_1 \longrightarrow 0$$
be the universal exact sequence on $B_1$, where $Q_1$ is the universal quotient bundle of rank $n_1$.
Observe that, as a complex vector bundle, $\mathcal{O}_{\mathbb{P}^r}(-1)$ coincides with $\xi_{1}^{\perp}$, the complement bundle, as defined in Section \ref{Sec_GMM}.

Assume that for some $1\le k<\ell$ we have constructed a complex manifold $B_k$ together with a holomorphic vector bundle $E_{k-1}$ of rank $n_k-k+2$ such that $B_k = \mathbb{P}(E_{k-1})$. Now let
    $$0 \longrightarrow \mathcal{O}_{B_k}(-1) \longrightarrow p_{k}^*E_{k-1} \longrightarrow Q_k \longrightarrow 0$$
be the universal exact sequence on $B_k$ where $p_{k} \colon B_k \to B_{k-1}$ and $Q_k$ is the universal quotient bundle of rank $(n_k-k+1)$.
As in the previous case, the bundle $\mathcal{O}_{B_k}(-1)$ is identified with $\xi_{k}$ and $Q_k$ is identified with $\xi_{k}^{\perp}$ as complex vector bundles.
We define a holomorphic vector bundle
    $$E_k := Q_k \oplus \mathcal{O}_{B_k}^{\oplus (n_{k+1}-n_k)},$$
which has rank $n_{k+1}-(k+1)+2$, and set $B_{k+1} := \mathbb{P}(E_k)$, with projection $p_{k+1}:B_{k+1}\to B_k$.
Thus we realize the generalized complex Milnor manifold $\mathbb{C}H(\mathbf{n})$ as an iterated projectivization of holomorphic vector bundles.
Since the projectivization of a holomorphic vector bundle over a K\"ahler manifold is again K\"ahler (see \cite[Proposition 3.18]{Voisin2002}), we obtain the following.

    \begin{proposition}\label{Prop_Kahler}
        The generalized Milnor manifold $\mathbb{C}H(\mathbf{n})$ is a compact K\"ahler manifold.
    \end{proposition}

In fact, by Serre's GAGA theorem (see \cite[Theorem 2.2.19]{Huy05}), there is a canonical bijection between the set of holomorphic vector bundles of rank $r$ and the set of locally free $\mathcal{O}_X$-modules of rank $r$. 
Using this correspondence, the above construction endows the generalized Milnor manifold with the structure of a smooth complex projective variety.

\section{Interplay between GMMs and Partial Flag Manifolds}\label{Sec_GMM and Flag mfld}

In this section, we address the question when a generalized Milnor manifold is homotopy equivalent to a partial flag manifold. We begin by recalling necessary background on partial flag manifolds and their cohomology, and treat the complex/quaternionic and real cases separately.

Let $\boldsymbol{\mu} = (m_1, m_2, \ldots, m_\wp)$ be a weakly increasing sequence of positive integers. Recall that the \emph{partial flag manifold} $\mathbb{F}G(\boldsymbol{\mu})$ is defined as:
$$
    \mathbb{F}G(\boldsymbol{\mu}) := \big\{V_1  \oplus \cdots \oplus V_\wp \mid V_i\perp V_j\text{ for }i\neq j\text{ and }\dim_{\mathbb{F}}V_i = m_i \text{ for } i=1,  \ldots, \wp\big\}.
$$
In particular, Grassmannians are partial flag manifolds of type $\mathbb{F}G(m_1, m_2)$, which we denote by $\mathbb{F}G_{m_1, m_1+m_2}$ alternatively.

Let $X$ be a finite CW complex and $\mathbb{K}$ a field. The Hilbert--Poincar\'e series of the cohomology algebra $H^*(X; \mathbb{K})$ is given by:
$$
    P_X(t) = \sum_{i \ge 0} \dim_{\mathbb{K}} H^i(X; \mathbb{K}) t^{i},
$$
where we set $\mathbb{K} = \mathbb{Z}_2$ if $\mathbb{F} = \mathbb{R}$, and $\mathbb{K} = \mathbb{Q}$ if $\mathbb{F} \in \{\mathbb{C}, \mathbb{H}\}$. Define $d := \dim_{\mathbb{R}}\mathbb{F}$ for $\mathbb{F} \in \{\mathbb{R}, \mathbb{C}, \mathbb{H}\}$. For the partial flag manifold $\mathbb{F}G(\boldsymbol{\mu})$ with $\boldsymbol{\mu}=(m_1,m_2,\ldots, m_\wp)$, one has $\dim H^d\bigl(\mathbb{F}G(\boldsymbol{\mu})\bigr) = \wp-1$.

We need the following lemma for future references.

\begin{lemma}\label{same mutisets}
    Let $\mathcal{T}_1$ and $\mathcal{T}_2$ be two towers of $\mathbb{F}$-projectivizations of heights $\ell_1$ and $\ell_2$, respectively, both starting from a trivial bundle over a point. Suppose that $\mathcal{T}_1$ and $\mathcal{T}_2$ are homeomorphic. Then $\ell_1 = \ell_2$, and the multisets of the fiber projective spaces appearing in the two towers coincide.
\end{lemma}
\begin{proof}
    Let $A = \{n_1, \dots, n_{\ell_1}\}$ and $B = \{m_1, \dots, m_{\ell_2}\}$ be the multisets of fiber dimensions for $\mathcal{T}_1$ and $\mathcal{T}_2$, respectively. For any projective bundle $\mathbb{F}P^k \hookrightarrow E \overset{}{\twoheadrightarrow} X$ in our towers, the cohomology of the fiber is generated by the first characteristic class of the canonical line bundle, which extends to a cohomology class of the total space $E$. By the Leray-Hirsch theorem, $H^*(E;\mathbb K) \cong H^*(X; \mathbb{K}) \otimes_{\mathbb{K}} H^*(\mathbb{F}P^k; \mathbb{K})$ as graded $\mathbb{K}$-modules.

    Iterating this decomposition up the towers, the Poincar\'e polynomials factor completely. Let $x = t^d$ and $Q_k(x) = \sum_{i=0}^k x^i$. Then the isomorphism of the cohomology rings dictates an equality of Poincar\'e polynomials in $\mathbb{Z}[x]$:
    \begin{equation}\label{eq:poincare_equality}
        \prod_{n \in A} Q_n(x) = \prod_{m \in B} Q_m(x).
    \end{equation}
    Equating the degree 1 coefficients of $x$ (which correspond to the $d$-th Betti numbers of the manifolds) yields $\ell_1 = \ell_2$.

    To show $A = B$ as multisets, recall that $Q_k(x)$ factors over $\mathbb{Z}[x]$ as $\prod \Phi_r(x)$ over all divisors $r \mid k+1$ with $r \ge 2$, where $\Phi_r(x)$ is the $r$-th cyclotomic polynomial. Because $\mathbb{Z}[x]$ is a unique factorization domain and the cyclotomic polynomials are irreducible, the multisets of irreducible factors on both sides of Equation (\ref{eq:poincare_equality}) must coincide. 
    
    Notice that the strictly largest cyclotomic index $r$ for which $\Phi_r(x)$ divides $Q_k(x)$ is $r = k+1$. Therefore, the maximum element $N \in A \cup B$ is uniquely identified by the presence of the factor $\Phi_{N+1}(x)$. Canceling the corresponding factors of $Q_N(x)$ from both sides and proceeding by descending induction on the maximum remaining elements strictly forces $A = B$.
\end{proof}

\subsection{Complex and Quaternionic Generalized Milnor Manifolds}

In this subsection, we completely classify the generalized Milnor manifolds that are homotopy equivalent to a partial flag manifold over the fields $\mathbb{F} \in \{\mathbb{C}, \mathbb{H}\}$. This is achieved by comparing their respective cohomology rings, which remain isomorphic under homotopy equivalence.

In \cite{BHH83}, Broughton--Hoffman--Homer characterized K\"ahler classes in a partial flag manifold and gave a formula for the height of a $d$-dimensional cohomology class as follows.

\begin{theorem}[Theorem 2.3 \cite{BHH83}]\label{kahler class}
    A cohomology class $x = \sum_{i=1}^{\wp} a_i x_i \in H^d\big(\mathbb{F}G(\boldsymbol{\mu}); \mathbb{Z}\big)$ is a K\"ahler class if and only if the coefficients are strictly decreasing, i.e., $a_1 > a_2 > \dots > a_{\wp}$.
\end{theorem}

 \begin{theorem}[Theorem 3.1 \cite{BHH83}]\label{height}
     Given a cohomology class $x = \sum_{i=1}^\wp a_i x_i \in H^d(\mathbb{F}G(\boldsymbol{\mu}); \mathbb{Z})$, let $\{b_1, \ldots, b_r\}$ denote the set of distinct values among the coefficients $\{a_1, \ldots, a_\wp\}$. For each $j = 1, \ldots, r$, define $n'_j = \sum_{a_i=b_j} n_i$. Then the height of $x$ is $\sum_{1 \le p < q \le r} n'_p n'_q$.
\end{theorem}

We will also need the following combinatorial lemma to bound the heights.

\begin{lemma}\label{min height}
    For a natural number $n,$ let $\mathcal{P}(n)$ denote the set of all partitions of $n$ into natural numbers. Given a partition $\lambda=(n_1,\ldots,n_t)\in \mathcal P(n)$, let $A_\lambda$ be the quantity $\sum_{1\le i<j\le t}n_in_j.$ Then \[
    \min\{A_\lambda: \lambda\in \mathcal P(n)\}=A_{\lambda_0}, \quad\text{where }\lambda_0=(1,n-1).
    \]
\end{lemma}
\begin{proof}
Let $\lambda=(n_1,\ldots,n_t)\in\mathcal P(n)$, so that $n_i\ge 1$ and $\sum_{i=1}^t n_i=n$. Using
\[
\left(\sum_{i=1}^t n_i\right)^2
= \sum_{i=1}^t n_i^2 + 2\sum_{1\le i<j\le t} n_i n_j,
\]
we obtain $2A_\lambda=n^2-\sum_{i=1}^t n_i^2$. Thus minimizing $A_\lambda$ is equivalent to maximizing $\sum_{i=1}^t n_i^2$.

We claim that that the partition for which the maximum of $\sum_{i=1}^tn_i^2$ attains, must have exactly two parts. Suppose $t\ge 3$. Replacing two parts $n_1,n_2$ by their sum produces a new partition $$\lambda'=(n'_1,n'_2, \ldots, n'_{t-1})=(n_1+n_2,n_3,\ldots,n_t)$$ with
$(n_1+n_2)^2>n_1^2+n_2^2$. Here,  $\sum_{i=1}^{t-1} (n'_i)^2>\sum_{i=1}^t n_i^2$, a contradiction. Therefore $t=2$.

Consider $\lambda=(k,n-k)$ with $1\le k\le n-1$. Then
$k^2+(n-k)^2=n^2-2k(n-k)$, which is maximized when $k(n-k)$ is minimized, i.e., at $k=1$ or $k=n-1$. Hence $\sum n_i^2$ is maximized for $\lambda_0=(1,n-1)$, and consequently $A_\lambda$ is minimized at $\lambda_0$.
\end{proof}

\begin{proposition}\label{height lower bound} 
The height of a nonzero cohomology class $x \in H^d\big(\mathbb{F}G(1^\ell, n); \mathbb{Z}\big)$ is at least $n$.
\end{proposition}

\begin{proof}
    Consider $\mathbb{F}G(\boldsymbol{\nu})$, where $\boldsymbol{\nu} = (1^\ell, n)$, and let $x = \sum_{i=1}^{\ell+1} a_i x_i \in H^d(\mathbb{F}G(\boldsymbol{\nu}); \mathbb{Z})$. Let $\{b_1, \ldots, b_r\}$ be the set of distinct values among the coefficients $\{a_1, \ldots, a_\ell\}$, and define $n'_j = \sum_{a_i=b_j} n_i$ for $j = 1, \ldots, r$.
    
    Then $\mathbb{F}G(\boldsymbol{\nu})$ admits the structure of a fibre bundle over $\mathbb{F}G(\boldsymbol{\nu}_b)$, where $\boldsymbol{\nu}_b = (n'_1, \ldots, n'_r)$, with fibre given by a tower of iterated partial flag bundles. We denote the projection map by $p:\mathbb{F}G(\boldsymbol{\nu}) \to \mathbb{F}G(\boldsymbol{\nu}_b)$.

    Consider the class $y = \sum_{j=1}^r b_j y_j \in H^d(\mathbb{F}G(\boldsymbol{\nu}_b); \mathbb{Z})$. Since the coefficients $b_j$ are distinct, Theorem~\ref{kahler class} implies that $y$ is a K\"ahler class. Hence, height of $y$ is $D_b = \dim_{\mathbb{F}}(\mathbb{F}G(\boldsymbol{\nu}_b))$. By naturality of Chern classes, we have $p^*(y) = x$. Moreover, by the Leray--Hirsch theorem, $p^*$ is injective in cohomology. It follows that the height of $x$ is $D_b$.

    Finally, note that $D_b = \sum_{1 \le i < j \le r} n'_i n'_j = A_\lambda$, where $\lambda = (n'_1, \ldots, n'_r) \in \mathcal{P}(n)$ as in Lemma~\ref{min height}. The result now follows from Lemma~\ref{min height}.
\end{proof}

The following proposition characterizes when a complex or quaternionic generalized Milnor manifold is homotopy equivalent to a partial flag manifold. 

\begin{proposition}\label{GMM neq flag}
    Let $\mathbf{n} = (n_1, n_2, \ldots, n_\ell)$ and $\boldsymbol{\mu} = (m_1, m_2, \ldots, m_\wp)$ be two weakly increasing sequences of positive numbers and $\mathbb{F} \in \{\mathbb{C}, \mathbb{H}\}$. Then the generalized Milnor manifold $\mathbb{F}H(\mathbf{n})$ is homotopy equivalent to the partial flag manifold $\mathbb{F}G(\boldsymbol{\mu})$ if and only if:
    $$
        \ell = \wp - 1, \quad n_1 = n_2 = \cdots = n_\ell, \quad m_1 = m_2 = \cdots = m_{\wp-1} = 1, \quad \text{and} \quad m_\wp = n_\ell - \ell + 1.
    $$
\end{proposition}
\begin{proof}
    Set $y = t^d$, where $d = \dim_{\mathbb{R}}\mathbb{F}$. Since $\mathbb{F}H(\mathbf{n})$ is constructed via iterated projectivizations as in Construction~\ref{construction of GMM fibre bundle} with fibres $\mathbb{F}P^{n_i - i + 1}$ for $i=1, 2, \ldots, \ell$, using the Leray--Hirsch theorem, we have the Hilbert--Poincar\'e series:
    \begin{equation}\label{poincare poly for GMM}
        P_{\mathbb{F}H(\mathbf{n})}(t) = \prod_{i=1}^{\ell} (1+y+\cdots+y^{n_i-i+1}).
    \end{equation}

    The partial flag manifold $\mathbb{F}G(\boldsymbol{\mu})$ is constructed by iterated associated Grassmannian bundles:
    $$
        \mathbb{F}G(\xi_{\wp-1}^\perp) \to \cdots \to \mathbb{F}G(\xi_1^\perp) \to \mathbb{F}G(\xi_0^\perp) \to \{*\},
    $$
    where $\xi_0^\perp := \mathbb{F}^{m}$, $m = \sum_{i=1}^{\wp}m_i$, and $\xi_i^\perp$ denotes the canonical complement bundle over the $i$-th associated Grassmannian bundle $\mathbb{F}G_{m_i}(\xi_{i-1}^\perp)$, for $i=1, 2, \ldots, \wp-1$. For simplicity, denote the fibre $\mathbb{F}G_{m_i, m_i+\cdots+m_\wp}$ in the $i$-th associated Grassmannian bundle $\mathbb{F}G_{m_i}(\xi_{i-1}^\perp)$ by $\mathbb{F}G_i$.

    Now, using Leray--Hirsch Theorem, we can write the Hilbert--Poincar\'e series of the cohomology algebra $H^*(\mathbb{F}G(\boldsymbol{\mu}); \mathbb{Z})$ as:
    \begin{equation}\label{poincare poly for flag}
        P_{\mathbb{F}G(\boldsymbol{\mu})}(t) = \prod_{i=1}^{\wp-1} P_{\mathbb{F}G_i}(t) = \prod_{i=1}^{\wp-1} (1 + b_{i,1}y + b_{i,2} y^2 + \cdots),
    \end{equation}
    where each $b_{i,1} = 1$ and $b_{i,2} \ge 1$. Note that $b_{i,2} = 1$ if and only if the fibre $\mathbb{F}G_{m_i, m_i+\cdots+m_\wp}$ in $\mathbb{F}G_i$ is a projective space.

    The sufficient part of the proof follows from Remark~\ref{remarks}(i).

    For the necessary part of the proof, assume that $\mathbb{F}H(\mathbf{n})$ and $\mathbb{F}G(\boldsymbol{\mu})$ are homotopy equivalent. It follows that their cohomology algebras are isomorphic as graded algebras and consequently, the coefficients of $y^i$ in \eqref{poincare poly for GMM} and \eqref{poincare poly for flag} coincide for all $i$. Comparing the coefficients of $y = t^d$ in \eqref{poincare poly for GMM} and \eqref{poincare poly for flag}, we obtain $\ell = \wp - 1$.

    If at least two of the $m_i$ in $\boldsymbol{\mu}$ exceed $1$, then some $b_{i,2} \ge 2$. This follows because $b_{i,2}$ counts the number of Schubert cells of real dimension $2d$, which corresponds to the number of integer partitions of $2$ fitting into an $m_i \times (\sum_{j>i} m_j)$ rectangle; whenever $m_i , \sum_{j>i} m_j\ge 2$, both partitions $(2)$ and $(1,1)$ are valid, yielding $b_{i,2} \ge 2$. Hence, it follows that the coefficient of $y^2$ in $P_{\mathbb{F}G(\boldsymbol{\mu})}(t)$ is strictly greater than the coefficient of $y^2$ in $P_{\mathbb{F}H(\mathbf{n})}(t)$, which is a contradiction. 
    Thus, at most one $m_i > 1$. Since $\boldsymbol{\mu}$ is weakly increasing, we have $m_1 = \cdots = m_{\wp-1} = 1$. In this case,
    \begin{equation}\label{new poincare poly of flag}
        P_{\mathbb{F}G(\boldsymbol{\mu})}(t) = \prod_{i=1}^{\ell}(1+y+\cdots+y^{m-i}),
    \end{equation}
    which agrees with $P_{\mathbb{F}H(\mathbf{n})}(t)$ if and only if the multisets $S_H := \{n_1, n_2-1, \ldots, n_\ell-(\ell-1)\}$ and $S_G := \{m-1, m-2, \ldots, m-\ell\}$ coincide, using Lemma~\ref{same mutisets}.

    It remains to show that $S_H$ and $S_G$ agree as ordered multisets, that is, $n_i - i + 1 = m - i$ for all $i$. Consequently, it will imply $n_1 = n_2 = \cdots = n_\ell = m-1$.

    Note that different orderings of $S_H$ give rise to different towers of projectivizations in the construction of a generalized Milnor manifold. Moreover, not every ordering yields such a construction. An ordering of $S_H$ produces a valid tower of projectivizations if and only if it satisfies:
    \begin{equation}\label{ordering condition}
        \text{\big($i$-th term\big)} - \text{\big($(i+1)$-th term\big)} \ge 1 \quad \text{for all } i.
    \end{equation}
   
    If an ordering of $S_H$ begins with the largest entry $m-1$, then the only way to extend it to a full ordering satisfying \eqref{ordering condition} is the ordered set $S_G$. In this case, the resulting generalized Milnor manifold is the partial flag manifold $\mathbb{F}G(1^\ell, n_\ell-\ell+1)$.

    Now suppose that the ordered multiset $S_H$ does not begin with $m-1$, i.e., $n_1 \neq m-1$, so that $n_1 < m-1$. It follows that the generator $x_1$ in the cohomology algebra, as described in Theorem~\ref{cohom of GMM}, has height $n_1 < m-1$. This leads to a contradiction, since there is no degree $d$ nonzero cohomology class $x \in H^d\big(\mathbb{F}G(1^\ell, n); \mathbb{Z}\big)$ such that $x^{n_1} \neq 0$ but $x^{n_1+1} = 0$, using Proposition~\ref{height lower bound}.
    Therefore, $S_H = S_G$ as ordered multisets, and this completes the proof.
\end{proof}

\subsection{The Real Case}

In the following proposition, we consider real generalized Milnor manifolds and obtain a weaker conclusion than in Proposition~\ref{GMM neq flag}.

\begin{proposition}\label{prop_GMM=flag}
    Let $\mathbf n=(n_1,n_2,\ldots, n_\ell)$ and ${\boldsymbol{\mu}}=(m_1,m_2,\ldots, m_\wp)$ be two weakly increasing sequences of positive integers. Let us consider the following statements:
   \begin{enumerate}[(1)]
    \item The generalized Milnor manifold $\mathbb{R}H(\mathbf{n})$ is homotopy equivalent to the partial flag $\mathbb{R}G(\boldsymbol{\mu})$.

    \item $m_1 = m_2 = \cdots = m_{\wp-1} = 1,\quad \wp=\ell+1, \quad  
    m_\wp = n_\ell - \ell + 1,\quad \text{and }
    n_1 = n_2 = \cdots = n_\ell.$

    \item $m_1 = m_2 = \cdots = m_{\wp-1} = 1,$ and 
    $\{m-i\}_{i=1}^\ell=\{n_i - i + 1\}_{i=1}^\ell$ as multisets.
\end{enumerate}
    Then, \textit{(1)}$\implies$\textit{(3)}, \textit{(2)}$\implies$\textit{(3)}, and  \textit{(2)}$\implies$\textit{(1)}. 
   \end{proposition}

\begin{proof}
The proof of \textit{(1)}$\implies$\textit{(3)} proceeds in the same way as in Theorem~\ref{GMM neq flag} to prove $S_H = S_G$ as unordered multisets.

The implications \textit{(2)}$\implies$\textit{(3)} and \textit{(2)}$\implies$\textit{(1)} are immediate.
\end{proof}

In the above proposition, it is natural to seek a proof of \textit{(1)}$\implies$\textit{(2)} rather than \textit{(1)}$\implies$\textit{(3)}. While we were not able to establish this stronger implication, we obtained the weaker conclusion \textit{(1)}$\implies$\textit{(3)}. Nevertheless, a restricted result holds in the case of real Milnor manifolds $\mathbb RH(n_1,n_2)$, which we describe in the following proposition.

\begin{proposition}\label{milnor not flag}
    In Proposition~\ref{prop_GMM=flag}, if we restrict to $\mathbb{F}=\mathbb{R}$ and $\ell=2$, i.e., to the case of real Milnor manifolds, then \textit{(1)}$\iff$\textit{(2)}.
\end{proposition}

\begin{proof}
         By Proposition~\ref{prop_GMM=flag}, we have \textit{(1)}$\implies$\textit{(3)}. We show that for real classical Milnor manifolds, ${\textit{(1)}\implies\textit{(2)}}$. The multisets in \textit{(3)} are $\{m-1,m-2\}$ and $\{n_1, n_2-1\}$. Now, there are two possibilities:
         \begin{align}
             \text{(i) } n_1=m-1,\quad n_2-1=m-2,\quad \quad \quad \text{(ii) }n_1=m-2,\quad n_2-1=m-1.
         \end{align}

    For the possibility (i), it is immediate that \textit{(3)}$\implies$\textit{(2)} and we are done.

    If possibility (ii) holds, we have $n_1=m-2$ and $n_2=m.$ Thus, $n_1, n_2$ have the same parity. Using Remark~\ref{remarks}~(i), we note that $\mathbb RG(1,1, m-2)$ is homeomorphic to $\mathbb RH(m-1,m-1)$. We need to compare $\mathbb RH(m-1,m-1)$ and $\mathbb RH(m-2,m)$.
    
    It is known (see \cite{HKM04}) that a real Milnor manifold $\mathbb{R}H(p,q)$ is orientable if and only if both $p$ and $q$ are even. Hence, exactly one of $\mathbb{R}H(m-1,m-1)$ and $\mathbb{R}H(m,m+2)$ is orientable, while the other is non-orientable. Therefore, they cannot be homotopy equivalent, and consequently possibility (ii) cannot occur.

    Therefore, here only possibility is (i) and this completes the proof.
\end{proof}

\section{Interpolation between partial flag manifolds via GMM}\label{Sec_Interpolation}
We have already observed in Proposition~\ref{prop_GMM=flag} the  situations in which generalized Milnor manifolds coincide with partial flag manifolds. In those cases, they are partial flag manifolds $\mathbb FG(\boldsymbol{\nu})$ of type $\boldsymbol{\nu}=(1^\ell,n)$.

Here we describe a family of generalized Milnor manifolds that interpolate between two partial flag manifolds of this type. More precisely, starting from the generalized Milnor manifold $\mathbb FH(n_1,n_2,\ldots, n_\ell)$, we obtain a sequence of generalized Milnor manifolds by modifying the vector bundles appearing in the successive projectivizations. These manifolds fit into a chain of natural embeddings of intermediate spaces between the two partial flag manifolds of type $(1^\ell, n)$.
Further, if we consider two different partial flag manifolds $\mathbb FG(\boldsymbol{\nu}_1)$ and $\mathbb FG(\boldsymbol{\nu}_2)$, where $\boldsymbol{\nu_1}=(1^\ell,p)$ and $\boldsymbol{\nu_2}=(1^\ell,q)$
    with $p<q,$ one can construct a sequence of generalized Milnor manifolds that fit into the natural embedding $\mathbb FG(\boldsymbol{\nu}_1)\hookrightarrow \mathbb FG(\boldsymbol{\nu}_2)$.

Using the Definition~\ref{defn of GMM equations} of generalized Milnor manifolds, we observe the following:
\begin{remark}\label{n<m FH(n)<FH(m)}
For two weakly increasing sequences $\mathbf n=(n_1,n_2,\ldots, n_\ell)$ and $\mathbf m=(m_1,m_2,\ldots, m_\ell)$ such that $n_i\le m_i$ for $ i=1,2\ldots ,\ell$, there is a natural embedding $\mathbb FH(\mathbf n)\hookrightarrow \mathbb F H(\mathbf m)$. \qed   
\end{remark}

As a consequence of Remark \ref{n<m FH(n)<FH(m)}, one obtains the following sequence of embeddings: (Boxes are used for emphasis only.)
\begin{equation}\label{combined_inclusions}
\mathbb FH(\mathbf n_1)\hookrightarrow\cdots \hookrightarrow \mathbb FH(\mathbf n_{\ell-1})
\hookrightarrow  \boxed{\mathbb FH(\mathbf n)}
\hookrightarrow \mathbb FH(\mathbf n_1') \hookrightarrow \cdots \hookrightarrow \mathbb FH(\mathbf n_{\ell-1}'),
\end{equation} where
\begin{align}\label{sequences}
\begin{aligned}
        &\mathbf n_{\ell-1}=\big(n_1,n_2,\ldots,n_{\ell-2},\boxed{n_{\ell-1}, n_{\ell-1}}\big),\\
        &\mathbf n_{\ell -2}=\big(n_1,n_2,\ldots,\boxed{n_{\ell-2},n_{\ell-2}, n_{\ell-2}}\big),\\
        &\quad\quad \vdots\\
        &\mathbf n_{1}=\big(\boxed{n_1,n_1,\ldots, n_{1},n_{1},n_{1}}\big),
\end{aligned}
\quad \text{and}\quad
\begin{aligned}
         &\mathbf n_{1}'=\big(\boxed{n_2,n_2},n_3, \ldots, n_{\ell}\big),\\
        &\mathbf n_{ 2}'=\big(\boxed{n_3,n_3,n_3}, \ldots, n_{\ell}\big),\\
        &\quad\quad \vdots\\
        &\mathbf n_{\ell-1}'=\big(\boxed{n_\ell,n_\ell, n_\ell,\ldots, n_{\ell}}\big).
\end{aligned}
\end{align}

     Using the Remark \ref{remarks} (ii), we have the following homeomorphisms: 
    \begin{align*}
        \mathbb FH(\mathbf n_{\ell -1}')\cong \mathbb FG(\boldsymbol{\nu^+}), \quad \text{where }\boldsymbol{\nu^+}=(1^\ell,\, n_\ell-\ell+1) & \text{ and }&\\
        \mathbb FH(\mathbf n_{1})\cong \mathbb FG(\boldsymbol{\nu^-}),\quad\text{where }\boldsymbol{\nu^-}=(1^\ell,\, n_1-\ell+1)
        .& &
    \end{align*}

    The sequence of embeddings in~\eqref{combined_inclusions} shows that the partial flag manifolds $\mathbb FH(\mathbf n_1)$ and $\mathbb FH(\mathbf n'_{\ell-1})$ are interpolated by a sequence of embeddings of generalized Milnor manifolds, with $\mathbb FH(\mathbf n)$ lying between two partial flag manifolds as follows: 
    \begin{equation}\label{flag GMM flag}
        \mathbb FG(\boldsymbol{\nu^-}) \lhook\joinrel \xrightarrow{\quad  \quad} \mathbb FH(\mathbf n)\lhook\joinrel\xrightarrow{\quad  \quad} \mathbb FG(\boldsymbol{\nu^+}).
    \end{equation}

    Now, let us consider two weakly increasing finite sequences of same length of the form:  
    \[
        \mathbf n=(n_1, n_2, \ldots, n_{i-1},n_i, n_{i+1},\ldots , n_{\ell})\text{ and }  \mathbf n^+=(n_1, n_2, \ldots, n_{i-1},n_i+1, n_{i+1},\ldots , n_{\ell}).
    \]

    Recall from Definition~\ref{construction of GMM fibre bundle} that $\mathbb FH(\mathbf n)$ and $\mathbb FH(\mathbf n^+)$ are constructed as iterated projectivizations. The generalized Milnor manifold $\mathbb FH(\mathbf n^+)$ is obtained from $\mathbb FH(\mathbf n)$ by replacing the $i$-th projectivization $\mathbb P(\xi_{i-1}^\perp \oplus \varepsilon_{\mathbb F}^{\,n_i-n_{i-1}})$ with $\mathbb P(\xi_{i-1}^\perp \oplus \varepsilon_{\mathbb F}^{\,n_i+1-n_{i-1}})$, while leaving all other stages unchanged. This yields a natural embedding
    \begin{equation}\label{inclusion: i to i+1}
        \iota\colon \mathbb FH(\mathbf n)\hookrightarrow \mathbb FH(\mathbf n^+).
    \end{equation}

    For each stage $j$, the canonical line bundles $\xi_j$ and their complementary bundles $\xi_j^\perp$ on $\mathbb FH(\mathbf n^+)$ pull back under $\iota$ to the corresponding bundles on $\mathbb FH(\mathbf n)$, except at the $(i+1)$-th stage. At this stage, we have $\iota^*(\xi_{i+1}) = \xi_{i+1}$ and $\iota^*(\xi_{i+1}^\perp) \cong \xi_{i+1}^\perp \oplus \varepsilon_{\mathbb F}$.

    By Proposition~\ref{cohom of GMM}, the cohomology algebra is generated by the characteristic classes of the bundles $\xi_j$ and $\xi_j^\perp$ (appropriate to $\mathbb F$). It follows, by naturality of characteristic classes, that the induced map $\iota^*$ is surjective on cohomology with $R$-coefficients. Moreover, $i^*$ is an isomorphism on $d$-th cohomology, where $d:=\dim_{\mathbb R}(\mathbb F)$. Consequently, each inclusion in~\eqref{combined_inclusions} induces an epimorphism in cohomology algebra and an isomorphism on $d$-th cohomology group in $R$-coefficients. 

    Observe that the embedding in Remark~\ref{n<m FH(n)<FH(m)} is a composition of inclusions of the type described in~\eqref{inclusion: i to i+1}. Thus, the embedding induces an epimorphism in $R$-cohomology.

    Additionally, we note that the inclusion $\mathbb FH(\mathbf n)\hookrightarrow \mathbb F H(\mathbf m)$ in Remark~\ref{n<m FH(n)<FH(m)} induces isomorphism in the $d$-th cohomology groups, where $d=\dim_{\mathbb R}(\mathbb F)$. It follows that every $\mathbb F$-line bundle over $\mathbb FH(\mathbf n)$ extends to $\mathbb FH(\mathbf m)$. Consequently, the isomorphism classes of $\mathbb F$-line bundles over $\mathbb FH(\mathbf n)$ and $\mathbb FH(\mathbf m)$ spaces are in bijection.

    In the following propositions, we investigate the optimality of the embedding $\mathbb FG(\boldsymbol{\nu}_1)\hookrightarrow \mathbb FH(\mathbf n)$. More precisely, we ask whether the partial flag manifold $\mathbb FG(\boldsymbol{\nu}_1)$ has the largest possible dimension among all partial flag manifolds that embed into $\mathbb FH(\mathbf n)$ such that the inclusion induces an epimorphism  in cohomology and an isomorphism in $d$-th cohomology.

\begin{proposition}\label{Prop_interpolation}
    Let $\mathbb FH(\mathbf{n})$ be a generalized Milnor manifold associated to a weakly increasing sequence $\mathbf{n}=(n_1,n_2,\ldots,n_\ell)$ and $\mathbb F\in \{\mathbb C, \mathbb H\}$. Consider an embedding
    \begin{equation}\label{inclusions}
        \mathbb{F}G(\boldsymbol{\nu}_1)\hookrightarrow \mathbb FH(\mathbf{n})
    \end{equation}
    with the inclusion inducing epimorphism in cohomology and isomorphism in $d$-th cohomology. Then     \[
        \dim\big(\mathbb FG(\boldsymbol{\nu}_1)\big)\le \dim\big(\mathbb FH(\mathbf{n}_1)\big)
        \quad\text{where }\mathbf n_{1}=\big(n_1^\ell\big)\text{ as in }\eqref{sequences}.
    \]
\end{proposition}

\begin{proof}
    Note that a partial flag manifold can be realized as an iterated tower of Grassmannian bundles, each associated to the preceding canonical complement bundle. Each such Grassmannianization contributes one generator to the $d$-th cohomology with $R$-coefficients, where $d:=\dim_{\mathbb{R}}(\mathbb{F})$, increasing the rank of $d$-th cohomology by 1.

    Since the inclusions in \eqref{inclusions} induce isomorphisms on $d$-th cohomology, it follows that the heights of the towers corresponding to $\mathbb FG(\boldsymbol{\nu}_1)$ and $\mathbb FG(\boldsymbol{\nu}_2)$ coincide with that of $\mathbb FH(\mathbf{n})$.

    Proceeding as in the proof of Proposition~\ref{GMM neq flag}, we deduce that the tower of associated Grassmannian bundles for $\mathbb FG(\boldsymbol{\nu}_1)$ can consist only of projectivizations. Indeed, if a  Grassmannianization stage were present which is not a projectivization, then the $2d$-th Betti number of $\mathbb FG(\boldsymbol{\nu}_1)$ would exceed that of $\mathbb FH(\mathbf{n})$, contradicting the fact that the inclusion $\mathbb FG(\boldsymbol{\nu}_1)\hookrightarrow \mathbb FH(\mathbf{n})$ induces an epimorphism in cohomology.
    Thus, the partial flag manifold $\mathbb FG(\boldsymbol{\nu}_1)$ corresponds to a sequence $\boldsymbol{\nu}_1$ of the form
    \[
        \boldsymbol{\nu}_1=(1^\ell,\, p-\ell+1)
        \quad \text{for some } p>\ell-1.
    \] Consequently, $\mathbb FG(\boldsymbol{\nu_1})$ becomes homeomorphic to $\mathbb FH(\mathbf p)$, where $\mathbf p=(p^\ell)$.

    Now, it remains to show that $p \le n_1$. Suppose, for contradiction, that $p > n_1$. Recall from Theorem~\ref{cohom of GMM} the presentation of the cohomology ring $H^*(\mathbb FH(\mathbf n);R)$. Note that, the generator $x_1 \in H^*(\mathbb FH(\mathbf n);R)$ has height $n_1$.

    On the other hand, by Theorem~\ref{height} together with Lemma~\ref{min height}, every nonzero element of positive degree in $H^*(\mathbb FG(\boldsymbol{\nu}_1);R)$ has height at least $p$. Since $p > n_1$, this contradicts the existence of the induced epimorphism in cohomology arising from the inclusion
    $ \mathbb FG(\boldsymbol{\nu}_1)\hookrightarrow \mathbb FH(\mathbf n).$
    Therefore, we must have $p \le n_1$, which completes the proof.
\end{proof}


\section{Higher Topological Complexity of GMM}\label{Sec_TC}

Farber \cite{Far03} introduced the notion of topological complexity $(\TC)$ to measure the inherent discontinuity present in any motion planning algorithm operating in a configuration space. The definition is motivated by the notion of Schwarz genus of fibrations, see \cite{Sva66}. This invariant has become an interesting area of study for topologists in recent years \cite{Rud10, Dav18, GM20, NS20}. In our work, we computed the exact value of the higher topological complexity of complex and quaternionic generalized Milnor manifolds. We also come up with the bounds of topological complexity in real case.

The \emph{sectional category} of a Hurewicz fibration $p: E\to B$, denoted by $\secat(p)$, is defined as the least non-negative integer $k$ such that there exists an open cover $\{U_0, \dots, U_k\}$ of $B$, where each open set $U_i$ admits a continuous section $s_i \colon U_i \to E$ of $p$. If no such $k$ exists, set $\secat(p)=\infty$.

Let $X$ be a path-connected space and $X^I$ be the free path space 
equipped with compact-open topology. Consider the fibration $$\pi \colon X^I \to X \times X, \quad \text{ defined by } \pi(\alpha)=(\alpha(0), \alpha(1)).$$ 
The \emph{topological complexity} $\TC(X)$ of $X$ is defined as $\TC(X) := \secat(\pi)$, see \cite{Far03}. 

Rudyak \cite{Rud10} generalized the notion of $\TC$ as follows.
For an integer $r \geq 2$, the \emph{$r$-th topological complexity} of a path-connected space $X$, denoted by $\TC_r(X)$, defined as the Schwarz genus of the fibration 
    \begin{equation}\label{Eq_fibration Pi r}
        \Pi_r \colon X^{I_r} \to X^r \quad \text{defined as }\Pi_r(\alpha)=(\alpha(1_1), \dots, \alpha(1_r)),
    \end{equation}
i.e. $\TC_r(X):=\secat (\Pi_r)$, where $I_r$ is the wedge of $r$ closed intervals $[0,1]$, each with $0$ as the base point  and $1_i$ stands for $1$ in the $i$-th interval.
Note that $\TC_2(X)=\TC(X)$.

Another topological invariant named \textit{LS-category}, denoted by $\cat$, was introduced by Lusternik-Schnirelmann \cite{LS34}. For a topological space $X$, $\cat(X)$ denotes the smallest positive integer $\mu$ such that $X$ can be covered by $\mu$ open subsets $V_0, V_1, \dots, V_{\mu}$ such that $V_i \hookrightarrow X$ is null-homotopic for $i=0, \dots, \mu.$ This invariant also provides some bounds for topological complexities. We recall from \cite[Corollary 3.3]{BGRT14} that for any path-connected space $X$, and natural number $r(\geq 2)$, 
    \begin{equation}\label{Eq_cat bounds TC}
        \cat(X^{r-1}) \leq \TC_r (X) \leq \cat (X^{r}).
    \end{equation}

Let us first fix some notation. Let $\mathbb K$ denote a field. Unless otherwise stated, $\mathbb K=\mathbb Q$ in the complex and quaternionic settings, while $\mathbb K=\mathbb Z_2$ in the real setting.

Now we want to focus on finding the lower bounds of these numerical homotopy invariants from cohomology algebra and dimension argument.
The cup-length of a topological space $X$, denoted by ${\cp} (X)$ is the maximal number $k$ such that there exists $x_i \in H^*(X;\mathbb K)$ for $i=1, \dots, k$ satisfying $\prod_{i=1}^k x_i \neq 0$, where each $\deg (x_i)>0$. Now we recall the upper and lower bounds of $\cat (X)$ for a path-connected space $X$ from \cite{CLOT03}
    \begin{equation}\label{Eq_upper and lower bound for cat}
        {\cp} (X) \leq \cat (X) \leq \frac{\hdim(X)}{\conn (X)+1},
    \end{equation}
where $\conn(X)$ is the largest degree up to which homotopy groups of $X$ are trivial and $\hdim(X)$ is the homotopy dimension of the topological space $X$.
Since these invariants are homotopy invariants, we are interested in homotopy dimensions which refer to the smallest dimensions of CW complexes of homotopy type $X$.

The fibration in \eqref{Eq_fibration Pi r} can be substituted by a iterated diagonal map $\Delta_r \colon X \to X^r$. 
Computing topological complexities is a challenge in general. In most cases, rather than determining the exact value, one focuses on finding bounds using various tools available in the literature. The upper bound is usually obtained from homotopy dimension arguments, while the lower bound comes from zero-divisor-cup-length, denoted by $\zcl_r(X)$. 
Let 
    $$\smile \colon \bigotimes_{i=1}^r H^*(X;\mathbb K) \to H^*(X;\mathbb K)$$ 
be the cup product map.  Then $\zcl_r (X)$ is defined to be the maximal number $\kappa$ such that there exist elements $u_i \in \ker \Delta_r^*$ for $i=1, \dots, \kappa$ satisfying $u_1 \smile u_2\smile \dots  \smile u_{\kappa} \neq 0$ in $H^*(X^r; \mathbb K)$, where $\Delta^*_r \colon H^*(X^r;\mathbb K) \to H^*(X;\mathbb K)$ is induced in cohomology by the diagonal map $\Delta_r$. The elements $u_i \in \ker \Delta_r^*$ are called the zero divisors.
The bounds
    \begin{equation}\label{Eq_upper and lower bound of TC by Schwarz}
        \zcl_r(X) \leq \TC_r(X) \leq \cat (X^r) \leq \frac{r \cdot \hdim(X)}{\conn(X)+1}
    \end{equation} 
of ${\TC}_r$ is due to Schwarz \cite[Theorems 4 and 5]{Sva66} where $\conn(X)$ is the connectivity of $X$ and $\hdim(X)$ is the homotopy dimension of the topological space $X$.
We recall the following from Basabe-Gonz\'alez-Rudyak-Tamaki\cite{{BGRT14}} for future use.
\begin{proposition}\cite[Corollary 3.15]{BGRT14}\label{Prop_tc of symplectic mfld}
    For any closed simply connected symplectic manifold $M$ of real dimension $2d$, we have $\TC_r(M)=rd$.
\end{proposition}

We begin with the following proposition where we calculate the exact value of LS-category and primary bounds of topological complexity of generalized Milnor manifolds.

\begin{proposition}\label{Prop_cat of GMM}
    For any $r \in \mathbb{N}$, we have 
        $$\cat \big( \mathbb{F}H(\mathbf{n})^r \big)=rD$$
    where $D$ is the $\mathbb{F}$-dimension of $\mathbb{F}H(\mathbf{n})$. As a consequence, for $r\geq 2$, we have 
        $$(r-1)D \leq {\TC}_r\big( \mathbb{F}H(\mathbf{n}) \big) \leq rD.$$
\end{proposition}

\begin{proof}
    We recall the cohomology algebra of $\mathbb{F}H(\mathbf{n})$ from Theorem \ref{cohom of GMM}. Using Leray-Hirsch theorem, we observe that $(x_{1})^{n_1} (x_{2})^{n_2-1} \cdots (x_{\ell})^{n_{\ell}-\ell+1}$ is a generator of $H^D \big( \mathbb{F}H(\mathbf{n});\mathbb K \big)$ where $x_{i}$ denotes the appropriate characteristic class as discussed in Theorem \ref{cohom of GMM}. Thus the product is non-zero in $H^* \big( \mathbb{F}H(\mathbf{n});\mathbb K\big)$.
    
    Now we consider the cohomology class
        \begin{equation}\label{Eq_cup length GMM}
        \prod_{j=1}^r (x_{1_j})^{n_1} (x_{2_j})^{n_2-1} \cdots (x_{\ell_j})^{n_{\ell}-\ell+1} \in H^*\big( \mathbb{F}H(\mathbf{n})^r ;\mathbb K\big)
        \end{equation}
    where $x_{i_j}$ denotes the appropriate characteristic class arising from $j$-th factor in the product $H^*\big( \mathbb{F}H(\mathbf{n})^r ;\mathbb K\big)$.
    Being the product of non-zero cohomology classes arising from different factors, the expression in \eqref{Eq_cup length GMM} is non-zero in $H^*\big( \mathbb{F}H(\mathbf{n})^r ;\mathbb K\big)$, by K\"unneth formula. Thus by \eqref{Eq_upper and lower bound of TC by Schwarz}, 
        $$\cat \big( \mathbb{F}H(\mathbf{n})^r\big) \geq rD.$$
    The other inequality comes from the dimension argument in \eqref{Eq_upper and lower bound for cat}.
    The consequence on topological complexity follows from \eqref{Eq_cat bounds TC}.
\end{proof}

The exact values of higher topological complexities of complex generalized Milnor manifolds can be determined using the fact that they are compact K\"ahler manifolds. In this context, we have the following theorem.

\begin{theorem}
\label{Prop_higher TC of complex GMM}
The higher topological complexity of the generalized Milnor manifold is given by 
    $$
    \TC_r(\mathbb{F}H(\mathbf{n}))= rD \quad \text{for } \mathbb{F}=\mathbb{C}, ~\mathbb{H}
    $$
where $D$ is the $\mathbb{F}$-dimension of $\mathbb{F}H(\mathbf{n})$.
\end{theorem}

\begin{proof}
In Subsection \ref{Subsec_GMM is Kahler}, we deduced that the generalized complex Milnor manifold $\mathbb{C}H(\mathbf{n})$ is a K\"{a}hler Manifold and consequently a symplectic manifold. Then using Proposition \ref{Prop_tc of symplectic mfld} and Remark \ref{remarks} (iii), we conclude $\TC_r(\mathbb{C}H(\mathbf{n}))=rD$.

Using Theorem \ref{cohom of GMM}, one observes that the presentations of the cohomology rings of $\mathbb CH(\mathbf n)$ and $\mathbb HH(\mathbf n)$ are isomorphic. The isomorphism is given by
\[
\phi : H^*(\mathbb CH(\mathbf n);\mathbb Q)\to H^*(\mathbb HH(\mathbf n); \mathbb Q),\quad c_1(\xi_i)\mapsto p_1({\xi_i}_{\mathbb R})\text{ for each }i,
\] where $c_1(\xi_i)$ denotes the first Chern class of the complex vector bundle $\xi_i$ and $p_1({\xi_i}_{\mathbb R})$ denote the first Pontryagin class of the underlying real vector bundle ${\xi_i}_{\mathbb R}$ of the canonical quaternionic vector bundle $\xi_{i}$ over $\mathbb HH(\mathbf n).$
Note that the isomorphism $\phi$ is not degree-preserving; rather, it doubles the degree. Since the cup-length depends only on the multiplicative structure of the ring and not on its grading, it follows that the cup-lengths of $\mathbb CH(\mathbf n)$ and $\mathbb HH(\mathbf n)$ are equal.
\end{proof}

\begin{remark}
    Theorem~\ref{Prop_higher TC of complex GMM} generalizes Theorem~2.6 of \cite{DS25} and, as a special case, provides an alternative proof of it. \qed
\end{remark}

We have already obtained a lower bound for the topological complexities of generalized real Milnor manifolds in Proposition~\ref{Prop_cat of GMM}. We now investigate whether this bound can be improved. In this direction, and more generally for fibre bundles satisfying the Leray--Hirsch hypotheses, Proposition \ref{Prop_zcl in Leray Hirsch setup} provides a useful tool for obtaining lower bounds of zero-divisor-cup-length in terms of fibre and base space, (cf. \cite[page~25]{HK00}). To prove that we need the following lemma.

\begin{lemma}\label{ker onto ker}
Let  $F\xhookrightarrow{i} E \overset{p}{\twoheadrightarrow}B$ be a fiber bundle such that $i^*: H^*(E; \mathbb{K}) \to H^*(F; \mathbb{K}) $ is surjective. Let $r \ge 2$ be an integer, and let $\mathbb{K}$ be a field. Denote by $X^r$ the $r$-fold cartesian product $X\times\cdots \times X$, and let $\Delta_{r,X}: X \to X^r$ be the diagonal map. Then the restriction
\[
(i^{\times r})^*: \ker(\Delta_{r,E}^*) \to \ker(\Delta_{r,F}^*)\quad\text{is surjective.}
\]
\end{lemma}

\begin{proof}
Consider the commutative diagram of short exact sequences of $\mathbb{K}$-vector spaces:
\begin{equation}\label{diag: ker onto ker}
    \begin{tikzcd}
	0 & {\ker \Delta_{r,E}^*} & {H^*(E^r;\mathbb K)} & {H^*(E;\mathbb K)} & 0 \\
	0 & {\ker \Delta_{r,F}^*} & {H^*(F^r;\mathbb K)} & {H^*(F;\mathbb K)} & 0
	\arrow[from=1-1, to=1-2]
	\arrow[hook, from=1-2, to=1-3]
	\arrow["\psi"', from=1-2, to=2-2]
	\arrow["{\Delta_{r,E}^*}", from=1-3, to=1-4]
	\arrow["{(i^{\times r})^*}", from=1-3, to=2-3]
	\arrow[from=1-4, to=1-5]
	\arrow["{i^*}", from=1-4, to=2-4]
	\arrow[from=2-1, to=2-2]
	\arrow[hook, from=2-2, to=2-3]
	\arrow["{\Delta_{r,F}^*}", from=2-3, to=2-4]
	\arrow[from=2-4, to=2-5]
\end{tikzcd}
\end{equation}
where $\psi := (i^{\times r})^*|_{\ker(\Delta_{r,E}^*)}$ is the restricted map. 

Using the K\"unneth formula, we have 
\[
H^*(E^r; \mathbb{K}) \cong \bigotimes_{i=1}^r H^*(E; \mathbb{K}) \quad \text{and} \quad H^*(F^r; \mathbb{K}) \cong \bigotimes_{i=1}^r H^*(F; \mathbb{K}).
\]
Under these identifications, the map $(i^{\times r})^*$ corresponds precisely to $\bigotimes_{i=1}^r i^*$. Indeed, the tensor product of any finite number of surjective linear maps over a field remains surjective, the middle vertical map $(i^{\times r})^*$ is surjective.

Using Snake Lemma in the diagram \eqref{diag: ker onto ker} yields the exact sequence of kernels and cokernels:
\[
\ker((i^{\times r})^*) \xrightarrow{\varphi} \ker(i^*) \xrightarrow{\delta} \mathrm{coker}(\psi) \to \operatorname{coker}((i^{\times r})^*)
\]
where $\varphi$ is the restriction of the map $\Delta_{r,E}^*$ to $\ker((i^{\times r})^*)$, and $\delta$ is the connecting homomorphism. 

Since the map $(i^{\times r})^*$ is surjective, $\operatorname{coker}((i^{\times r})^*) = 0$. By exactness at $\operatorname{coker}(\psi)$, the sequence terminates as:
\[
\operatorname{ker}((i^{\times r})^*) \xrightarrow{\varphi} \operatorname{ker}(i^*) \xrightarrow{\delta} \operatorname{coker}(\psi) \to 0,
\]
which implies that $\delta$ is surjective. Thus, to show $\operatorname{coker}(\psi) = 0$, it suffices to prove that $\varphi$ is surjective.
Let $\alpha \in \ker(i^*) \subseteq H^*(E; \mathbb{K})$.  Consider the element 
\[
\tilde{\alpha} = \alpha \otimes 1 \otimes \dots \otimes 1 \in H^*(E^r; \mathbb{K}).
\]
This implies that $\Delta_{r,E}^*(\tilde{\alpha}) = \alpha \cdot 1 \cdots 1 = \alpha$, and consequently
\[
(i^{\times r})^*(\tilde{\alpha}) = i^*(\alpha) \otimes i^*(1) \otimes \dots \otimes i^*(1) = 0 \otimes 1 \otimes \dots \otimes 1 = 0,
\]
which proves that $\tilde{\alpha} \in \ker((i^{\times r})^*)$. Since $\varphi(\tilde{\alpha}) = \Delta_{r,E}^*(\tilde{\alpha}) = \alpha$, the map $\varphi$ is surjective.
\end{proof}

\begin{proposition}\label{Prop_zcl in Leray Hirsch setup}
    Let $p:E\to B$ be a fibre bundle with fibre $F$ such that $H^*(F;\mathbb K)$ is finite-dimensional and the inclusion of each fibre into $E$ induces a surjection in cohomology with coefficients in a field $\mathbb K$. Then 
    $$\zcl_r(F) + \zcl_r(B) \leq \zcl_r(E).$$
\end{proposition}

\begin{proof}
    Let $i:F\hookrightarrow E$ be a fibre inclusion and $i^*: H^*(E,\mathbb K)\to H^*(F;\mathbb K)$ is an epimorphism. By Leray--Hirsch theorem, the projection $p$ induces a monomorphism in cohomology.
    
    Using the fiber bundle $F\xhookrightarrow{i} E \overset{p}{\twoheadrightarrow}B,$  one can construct a new fibre bundle
    \begin{equation}\label{product fibre bundle}
        \begin{tikzcd}
	   {F^r} & {E^r} & {B^r}
	   \arrow["{i^{\times r}}", hook, from=1-1, to=1-2]
	   \arrow["{p^{\times r}}", two heads, from=1-2, to=1-3]
        \end{tikzcd}
    \end{equation}where $X^r$ denotes the $r$-fold cartesian product $X\times \cdots \times X$, and $i^{\times r},\; p^{\times r}$ denote the $r$-fold product maps $i\times \cdots \times i,\; p\times \cdots \times p$, respectively.
    Using K\"unneth formula we identify $H^*(X^r;\mathbb K)$ with $\otimes _{i=1}^{r}H^*(X;\mathbb K)$. Then we have the following commutative diagram
\begin{equation}\label{commutative diagram}
    \begin{tikzcd}
	{\bigotimes_{i=1}^r H^*(F;\mathbb K)} & {\bigotimes_{i=1}^r H^*(E,\mathbb K)} & {\bigotimes_{i=1}^r H^*(B;\mathbb K)} \\
	{H^*(F;\mathbb K)} & {H^*(E,\mathbb K)} & {H^*(B;\mathbb K)}
	\arrow["{\Delta_{r,F}^*}", from=1-1, to=2-1]
	\arrow["{\otimes i^*}"', two heads, from=1-2, to=1-1]
	\arrow["{\Delta_{r,E}^*}", from=1-2, to=2-2]
	\arrow["{\otimes p^*}"', hook', from=1-3, to=1-2]
	\arrow["{\Delta_{r,B}^*}", from=1-3, to=2-3]
	\arrow["{i^*}"', two heads, from=2-2, to=2-1]
	\arrow["{p^*}"', hook', from=2-3, to=2-2],
\end{tikzcd}
\end{equation}
where $\otimes i^*$ and $\otimes p^*$ are induced from $i^{\times r}$ and $p^{\times r}$, respectively.
The commutativity of the diagram follows from the fact that all the maps in the diagram are induced from topological maps and they commute appropriately.

Let $t:=\zcl_r(F;\mathbb K)$ and $s:=\zcl_r(B;\mathbb K).$ Let $P_F=f_1f_2\cdots f_t$ be a nontrivial product of $t$ positive-degree elements from $\ker(\Delta_{r,F}^*)$. Likewise, let $P_B=b_1b_2\cdots b_s$ be a nontrivial product of $s$ positive-degree elements from $\ker(\Delta_{r,B}^*)$.

Let $\{\alpha_1,\alpha_2,\ldots,\alpha_n\}$ be a $\mathbb K$-basis of $H^*(F^r;\mathbb K)$ with $\alpha_1=P_F$. Using Lemma~\ref{ker onto ker}, we can choose cohomology extensions $\tilde f_1,\tilde f_2,\ldots,\tilde f_t\in \ker \Delta_{r,E}^*$ under $\otimes i^*$ of $f_1,f_2,\ldots,f_t$, respectively. Choose cohomology extensions $\tilde\alpha_1,\tilde\alpha_2,\ldots,\tilde\alpha_n\in H^*(E^r;\mathbb K)$ of $\alpha_1,\alpha_2,\ldots,\alpha_n$ such that $i^*(\tilde\alpha_i)=\alpha_i$ for $1\le i\le n$ and $\tilde\alpha_1=\tilde f_1\tilde f_2\cdots\tilde f_t$.

Now using Leray--Hirsch theorem for the fibre bundle $\eqref{product fibre bundle}$, we know that $H^*(E^r;\mathbb K)$ is a free $H^*(B^r;\mathbb K)$-module and we have 
the $H^*(B^r;\mathbb K)$-module isomorphism given by
\[
\phi: H^*(B^r;\mathbb K)\otimes H^*(F^r;\mathbb K)  \longrightarrow H^*(E^r;\mathbb K), \quad b\otimes \sum_{i=1}^{n} a_i \alpha_i\;\xmapsto{\quad\phi\quad}\; p^*(b)\smile \sum _{i=1}^n a_i \tilde
\alpha_i.
\]
Under this isomorphism, $0\neq P_B\otimes P_F=P_B\otimes \alpha_1$ corresponds to $\otimes p^*(P_B)\smile\tilde \alpha_1\in H^*(E^r;\mathbb K)$, ensuring that $\otimes p^*(P_B)\smile\tilde \alpha_1$ is nonzero.

The fact that the product $\otimes p^*(P_B)\smile\tilde \alpha_1\in \ker \Delta _{r,E}^*$ is followed from the observation that 
\begin{align*}
    \Delta_{r,E}^*(\tilde \alpha_1\smile \otimes p^*(P_B))&=\Delta_{r, E}^*(\tilde \alpha_1 
    )\smile \Delta^*_{r,E}\big(\otimes p^*(P_B)\big)\\
    &=\Delta_{r, E}^*(\tilde \alpha_1)\smile \otimes p^*\big(\Delta^*_{r,B}(P_B)\big),\text{ using the commutativity of diagram~\eqref{commutative diagram}}\\
    &= 0, \text{ since }\Delta^*_{r,B}(P_B)=0.
\end{align*}

Since $\otimes p^*$ is a ring monomorphism, the class $\otimes p^*(P_B)$ is nonzero and factors as the product of $s=\zcl_r(B;\mathbb K)$ positive-degree elements, namely the images of $b_1,b_2,\ldots,b_s$ under $\otimes p^*$. Together with the fact $\tilde\alpha_1=\tilde f_1\tilde f_2\cdots\tilde f_t$, it follows that the nonzero class $\otimes p^*(P_B)\smile\tilde\alpha_1$ can be expressed as a product of at least $\zcl_r(B;\mathbb K)+\zcl_r(F;\mathbb K)$ many positive-degree cohomology classes. Therefore,
$\zcl_r(E;\mathbb K)\ge \zcl_r(B;\mathbb K)+\zcl_r(F;\mathbb K)$, completing the proof.
\end{proof}

There exist many articles exploring the zero-divisor-cup-length and consequently topological complexity of real projective space $\RR P^n$ and classical real Milnor manifold $\RR H(r,s)$, see \cite{Dav18, NJDL18,FTY03,ND25,DS25} for instance. We exploit these existing results to find better lower bounds than the bounds given in Proposition \ref{Prop_cat of GMM} for topological complexity of generalized real Milnor manifolds. Here, we recall the following result from Davis\cite{Dav18} for later use.

\begin{theorem}\cite[Theorem 1.6]{Dav18}\label{Thm_Davis}
    Let $n=2^t+s$, $t\geq 0$ and $0 \leq s < 2^t$. Then for $r\geq 2$, 
        \begin{equation}\label{Eq_Davis}
            \zcl_r(\RR P^n)=\min \{ \zcl_r(\RR P^d) + r 2^t, (r-1)(2^{t+1}-1) \}, \quad \text{where we set} ~\zcl_r(\RR P^0)=0.
        \end{equation}
\end{theorem}

In Proposition \ref{Prop_GMM as iterated fibre bundle}, we discussed how we can realize a generalized Milnor manifold as an iterated fibre bundle with each fibre a generalized Milnor manifold.
Let $\mathbb{R} H(\textbf{n})$ be a generalized Milnor manifold with $\textbf{n}=(n_1, \dots, n_{\ell})$.
Now let $\mathcal{P}_{\{1,2\}}(\ell)$ denote the set of all ordered partitions of $\ell$ whose parts belong to $\{1,2\}$, that is,
\begin{equation}\label{Eq_partition}
    \mathcal{P}_{\{1,2\}}(\ell):=\Big\{\lambda=(\lambda_1,\ldots,\lambda_k)\mid \lambda_i\in\{1,2\}\ \text{for all } i,\ \sum_{i=1}^k\lambda_i=\ell\Big\}.
\end{equation}
 Then we can express $\mathbb{R} H(\textbf{n})$ by an iterated fibre bundle as 
    \begin{equation}\label{Eq_iteration of GMM using 1 and 2}
        \mathbb{R} H(\textbf{n})=\mathbb{R} H(\textbf{n}_{\lambda_k}) \twoheadrightarrow \mathbb{R} H(\textbf{n}_{\lambda_{k-1}}) \twoheadrightarrow \cdots \twoheadrightarrow \mathbb{R} H(\textbf{n}_{\lambda_2}) \twoheadrightarrow \mathbb{R} H(\textbf{n}_{\lambda_1})\twoheadrightarrow \{*\}
    \end{equation}
using $\lambda \in \mathcal{P}_{\{1,2\}}(\ell)$. We have $\zcl_r(\mathbb{R} H(\textbf{n}_{\lambda_i}))$ from existing results in the literature. Using Proposition \ref{Prop_zcl in Leray Hirsch setup} iteratively in the tower of fibre bundles in \eqref{Eq_iteration of GMM using 1 and 2}, we get
    \begin{equation}\label{Eq_L lambda}
        \mathscr{L}_{\lambda}\big( \mathbb{R} H(\textbf{n}) \big) \coloneq \sum_{i=1}^k \zcl_r \big( \mathbb{R} H(\textbf{n}_{\lambda_i}) \big) \leq \zcl_r \big( \mathbb{R} H(\textbf{n}) \big),
    \end{equation}
a lower bound of $\TC_r(\mathbb{R} H(\textbf{n}))$. If we choose a different partition $\lambda' \in \mathcal{P}_{\{1,2\}}(\ell)$, we get another lower bound $\mathscr{L}_{\lambda'}(\mathbb{R} H(\textbf{n}))$. We conclude the above discussion with the following result.

\begin{proposition}\label{Prop_lower bound by partition}
    For any generalized real Milnor manifold $\RR H(\textbf{n})$  and $r \geq 2$,  
        $$\max \{ \mathscr{L}_{\lambda}(\mathbb{R} H(\textbf{n})) \mid \lambda \in \mathcal{P}_{\{1,2\}}(\ell)\} \;\leq \; \TC_r(\mathbb{R} H(\textbf{n}))\; \leq \; rD$$
    where $D=\dim(\mathbb{R} H(\textbf{n}))$ and $\mathcal{P}_{\{1,2\}}(\ell),\;\mathcal{L}_{\lambda}$ are defined as in \eqref{Eq_partition},\eqref{Eq_L lambda}, respectively.
\end{proposition}

\begin{corollary}\label{Cor_tc of real gmm}
    Let $\RR H (\textbf{n})$ be a real generalized milnor manifold with $n_i=2^{t_i}+i-1$ where $t_i \geq 0$ for $1 \leq i \leq \ell$. Then for $r=2$, we have
    \begin{equation}
        2D- \ell \leq \TC (\RR H(\textbf{n})) \leq 2D,
    \end{equation}
    and for $r \geq 3$, we get
        $$\TC_r(\RR H(\textbf{n}))=rD,$$
    where $D$ is the dimension of $\RR H(\textbf{n})$.
\end{corollary}

\begin{proof}
    The right inequality follows from Proposition \ref{Prop_cat of GMM}. Also notice that under our assumption, the fibre at every step is $\RR P^{2^{t_i}}$ for $i=1, \dots, \ell$. The lower bounds of $\TC$ is given by \eqref{Eq_Davis} as
        $$\zcl_2(\RR P^{2^{t_i}})=2^{t_i+1}-1 \quad \text{and} \quad \zcl_r(\RR P^{2^{t_i}})=r2^{t_i}.$$
    Then we can conclude the result using Proposition \ref{Prop_zcl in Leray Hirsch setup} and \ref{Prop_lower bound by partition}.
\end{proof}

Moreover, Propositions \ref{Prop_zcl in Leray Hirsch setup} and \ref{Prop_lower bound by partition} can be used, together with Theorem \ref{Thm_Davis}, to derive lower bounds for $\TC_r$ of any real generalized Milnor manifold. In general, this technique is better suited for the total spaces $E$ with base $B$ and fibre $F$ having $\zcl_r(B)$ and $\zcl_r(F)$ close to their dimensions $\dim(B)$ and $\dim(F)$ respectively.
Alternatively, constructing an explicit nonzero product of zero-divisors in the cohomology algebra may provide a better estimation of the zero-divisor-cup-length and eventually a better lower bound for topological complexity in certain scenarios. We illustrate this alternative approach with an example in which we calculate the lower bound of $TC_2$ of some real GMM.

\begin{example}
    In Corollary \ref{Cor_tc of real gmm}, we choose the GMM such a way that the fibres are $\RR P^{k}$ where $k$ is a power of two. 
    We consider $n_2=2^{t_2}$ instead of $2^{t_2}+1$ and $n_i=2^{t_i}+i-1$ for $i=1,3, \dots, \ell$.
    Recall from Remark \ref{remark heights of generators} that $x_i^{n_i} \neq 0$, $x_i^{n_i+1}=0$. For $i=1,\dots, \ell$, the classes $\bar{x}_i \coloneq x_{i_1} + x_{i_2} \in \ker \Delta_2^*$
    are zero divisors where $x_{i_1}=x_i \otimes 1$ and $x_{i_2}=1 \otimes x_i$.
    We now consider the product of zero divisors
    \begin{equation}\label{Eq_prod for zcl2}
       \bar{x}_2^{2n_2-2} \prod_{\substack{{i=1}\\{i \neq 2}}}^{\ell} \bar{x}_i^{2n_i-(2i-1)} \quad \in H^* \big( \RR H (\textbf{n})^2;\mathbb Z_2 \big)
    \end{equation}    
    and show that it is nonzero.
   Note that for each $i=1,3,\dots, \ell$, the product $\bar{x}_i^{2n_i-(2i-1)}=(x_{i_1}+x_{i_2})^{2n_i-(2i-1)}$ contains the term $x_{i_1}^{n_i-i+1} x_{i_2}^{n_i-i}$ as $2k-1 \choose k$ is odd for all $k=2^t \in \mathbb{N}$, by Lucas Theorem.
    Also $\bar{x}_2^{2n_2-2}$ contains a non-vanishing term $x_{2_1}^{n_2-2} x_{2_2}^{n_2}$ since $2n_2-2 \choose n_2$ is odd, again by Lucas theorem.
    Thus, the product in \eqref{Eq_prod for zcl2} contains the term $\prod_{i=1}^{\ell} x_{i_1}^{n_i-i+1} x_{i_2}^{n_i-i}$ only once and no other term cancels this. 
    Here we used the fact that $x_1^{n_1} x_2^{n_2-1}=x_1^{n_1-1}x_2^{n_2}$. Thus $$2D -\ell +1 \leq \TC (\RR H(\textbf{n})) \leq 2D,$$ where $D$ is the dimension of GMM.
\end{example}

\begin{remark}
Corollary \ref{Cor_tc of real gmm} improves the lower bounds for the topological complexity of classical real Milnor manifolds obtained in \cite[Theorem 2.9(2)]{DS25}. More significantly, Proposition \ref{Prop_zcl in Leray Hirsch setup} provides a general and flexible tool for estimating the zero-divisor-cup-length of total spaces arising in the Leray--Hirsch setting. In particular, it allows one to deduce lower bounds for the higher topological complexities of iterated Grassmannian bundles associated to vector bundles directly from the corresponding invariants of their base spaces and fibres.

Since many important classes of manifolds, including partial flag manifolds and generalized Bott manifolds admit descriptions as towers of Grassmannian bundles associated to vector bundles (see \cite{CMS10},\cite{KL03}), Proposition \ref{Prop_zcl in Leray Hirsch setup} furnishes a unified framework for studying the zero-divisor-cup-length and topological complexity of such spaces. \qed
\end{remark}

\subsection*{Acknowledgements}
We would like to thank Prof. A. S. Thakur for introducing us to Milnor manifolds and for insightful conversations. We are grateful to Prof. P. Sankaran for sharing his insights on holomorphic vector bundles over Milnor manifolds. Finally, we thank IIT Kanpur for its postdoctoral fellowship.

\bibliographystyle{amsalpha}
\bibliography{TC-ref}

\end{document}